\documentclass[11pt,reqno]{amsart}

\usepackage[a4paper, left=30mm,right=30mm,top=30mm,bottom=30mm,marginpar=25mm]{geometry}
\usepackage[english]{babel}
\usepackage{setspace}
\usepackage{amsmath}
\usepackage{mathtools}
\usepackage[normalem]{ulem}
\usepackage{amsfonts}
\usepackage{amssymb}
\usepackage{fancyhdr}
\usepackage{textcomp}
\usepackage{graphicx} 
\usepackage{float}
\usepackage{listings}
\usepackage{amssymb}
\usepackage{color}

\makeatletter
\def\tagform@#1{\maketag@@@{\normalsize(#1)\@@italiccorr}}
\makeatother

\newtheorem{theorem}{Theorem}%[section]
\newtheorem{lemma}[theorem]{Lemma}%[section]
\newtheorem{remark}[theorem]{Remark}%[section]
\newtheorem{corollary}[theorem]{Corollary}%[section]
\newcommand{\ds}{\displaystyle}
\newcommand{\ra}{\rangle}
\newcommand{\la}{\langle}
\newcommand{\lf}{\lambda_{F}}
\newcommand{\lz}{\lambda_{\zeta}}
\newcommand{\lw}{\lambda_{w}}
\newcommand{\lx}{\lambda_{\xi}}
\newcommand{\lv}{\lambda_{v}}
\newcommand{\lvi}{\lambda_{v_i}}
\newcommand{\li}{\lambda_{\eta}}
\newcommand{\fe}{F^\varepsilon}
\newcommand{\ve}{v^\varepsilon}
\newcommand{\ze}{\zeta^\varepsilon}
\newcommand{\we}{w^\varepsilon}
\newcommand{\xe}{\xi^\varepsilon}

\newcommand{\ee}{\eta^\varepsilon}

\newcommand{\bn}{\boldsymbol{\nu}}

\newcommand{\cO}{\mathcal{O}}

\newcommand{\bg}{\boldsymbol{\gamma}}

\newcommand{\ph}{\frac{\partial\Phi^B}{\partial F_{i\alpha}}(F)}
\newcommand{\phl}{\frac{\partial\Phi^B}{\partial F_{i\alpha}}(\lambda_F)}
\newcommand{\phb}{\frac{\partial\Phi^B}{\partial F_{i\alpha}}(\bar{F})}
\newcommand{\phv}{\frac{\partial\Phi^B}{\partial F_{i\alpha}}(F)v_i}

\newcommand{\phbv}{\frac{\partial\Phi^B}{\partial F_{i\alpha}}(\bar{F})\bar{v}_i}

\newcommand{\cof}{\hbox{cof}\,}
\newcommand{\dt}{\hbox{det}\,}
\newcommand{\phz}{\frac{\partial\Phi^B}{\partial F_{i\alpha}}(F^0)}
\newcommand{\pge}{\frac{\partial \hat{e}}{\partial \xi^B}(\xi,\eta)}
\newcommand{\phzj}{\frac{\partial\Phi^B}{\partial F_{i\alpha}}(F^{j-1})}
\newcommand{\pgej}{\frac{\partial \hat{e}}{\partial \xi^B}(\xi^j,\eta^j)}

\newcommand{\tcr}{}
\newcommand{\tcb}{}
\definecolor{listinggray}{gray}{0.9}
\definecolor{lbcolor}{rgb}{0.9,0.9,0.9}

\def\del{\partial}
\allowdisplaybreaks
\begin{document}
	
	\numberwithin{equation}{section}
	
	\title[A discrete variational scheme in polyconvex thermoelasticity]
	{A discrete variational scheme for isentropic processes in polyconvex thermoelasticity}
	
	\author[C. Christoforou]{Cleopatra Christoforou}
	\address[Cleopatra Christoforou]{Department of Mathematics and Statistics,
		University of Cyprus, Nicosia 1678, Cyprus.}
	\email{christoforou.cleopatra@ucy.ac.cy} 
	\author[M. Galanopoulou]{Myrto Galanopoulou}
	\address[Myrto Galanopoulou]{Computer, Electrical, Mathematical Sciences \& Engineering Division, King Abdullah University of Science and Technology (KAUST), Thuwal, Saudi Arabia.}
	\email{myrtomaria.galanopoulou@kaust.edu.sa}
	\author[A. E. Tzavaras]{Athanasios E. Tzavaras}
	\address[Athanasios E. Tzavaras]{Computer, Electrical, Mathematical Sciences \& Engineering Division, King Abdullah University of Science and Technology (KAUST), Thuwal, Saudi Arabia.}
	\email{athanasios.tzavaras@kaust.edu.sa}

	\begin{abstract}
		We propose a variational scheme for the construction of isentropic processes of the equations of
		adiabatic thermoelasticity with polyconvex internal energy. The scheme hinges on the embedding of the equations of adiabatic polyconvex 
		thermoelasticity  into a symmetrizable hyperbolic system. We establish existence of minimizers for an associated  minimization Theorem
		and construct measure-valued solutions that dissipate the total energy. We prove that the scheme converges when the limiting solution is smooth.
	\end{abstract}

	\maketitle
	%\tableofcontents
	
		\section{Introduction}\label{3:SEC1}
We consider the system of adiabatic thermoelasticity,
\begin{align}
\begin{split}
\label{e:adiabatic_thermoelasticity}
\partial_t F_{i\alpha}-\partial_{\alpha}v_i&=0 \\ 
\partial_t v_i-\partial_{\alpha}\Sigma_{i\alpha}&=0\\
\partial_t \eta&=\frac{r}{\theta}\\
\partial_t\left(\frac{1}{2}|v|^2+e\right)-\partial_{\alpha}(\Sigma_{i\alpha}v_i)&=r ,
\end{split}
\end{align}
describing the evolution of a thermomechanical process $\big ( y(x,t) , \eta(x,t) \big) \in \mathbb{R}^3\times\mathbb{R}^+$ \tcb{in Lagrangian coordinates}
with the spatial variable $x\in\mathbb{R}^3$ and time $t\in\mathbb{R}^+$. A solution to (\ref{e:adiabatic_thermoelasticity}) consists of the deformation gradient
$F = \nabla y  \in \mathbb{M}^{3\times3}$, the velocity $v = \partial_t y \in\mathbb{R}^3$ and the specific entropy
$\eta$. The first equation is a compatibility relation, the
second describes the balance of linear momentum, while the fourth stands for the balance of energy.
One appends to \eqref{e:adiabatic_thermoelasticity} the constraint 
\begin{equation}
\label{constraint}
\partial_{\alpha}F_{i\beta}=\partial_{\beta}F_{i\alpha}, \qquad i,\alpha,\beta=1,2,3 \, ,
\end{equation}
which guarantees that $F$ is indeed a gradient. Note that \eqref{constraint} is an involution, that is it is  propagated from the initial data to the solution via \eqref{e:adiabatic_thermoelasticity}$_1$.

The remaining variables in (\ref{e:adiabatic_thermoelasticity}) are the Piola-Kirchhoff stress  $\Sigma_{i\alpha }$, the internal energy $e$, 
and the radiative heat supply $r$.  
\tcb{ In the theory of adiabatic thermoelasticity, 
	the referential heat flux $Q_\alpha = 0$ and it does not appear in the equations \eqref{e:adiabatic_thermoelasticity}$_{3,4}$.}
For simplicity we have normalized the reference density $\rho_0 =1$.
The balance of entropy \eqref{e:adiabatic_thermoelasticity}$_3$ holds identically as an equality for strong solutions; by contrast, for weak solutions it 
is replaced by the Clausius-Duhem inequality \cite{CN1963,TN1992,MR3468916} and serves as an admissibility criterion.
The system is closed through constitutive 
relations which, for smooth processes, are consistent with the Clausius-Duhem inequality and describe the material response. 
For thermoelastic materials under adiabatic conditions,  the constitutive theory  is determined from the thermodynamic potential of the internal energy $e(F,\eta)$, via 
the relations
\begin{align}
\label{e:constitutive_functions}
e=e(F,\eta),\quad
\Sigma=\frac{\partial e}{\partial F},\quad
\theta=\frac{\partial e}{\partial\eta}
\end{align}
for the stress $\Sigma$ and the temperature $\theta$, respectively.  We refer the reader to \cite{CN1963,TN1992} for a detailed derivation of the theory of adiabatic thermoelasticity
and its relation with more elaborate constitutive theories.

Our objective with this work is to construct a variational approximation scheme in the spirit of \cite{MR1831179}, extending that analysis to
the full system of adiabatic thermoelasticity \eqref{e:adiabatic_thermoelasticity} under a general class of constitutive laws. We adopt the hypothesis of polyconvexity 
which for thermoelasticity asserts that the free energy $e(F,\eta)$ factorizes as	
\begin{equation*}
e(F,\eta)=\hat{e}(\Phi(F),\eta),\quad \text{where}\quad \Phi(F)=(F,\mathrm{cof}F,\det F)\in\mathbb{R}^{19}
\end{equation*}
and $\hat e : \mathbb{R}^{19} \times \mathbb{R^+} \to \mathbb{R}$ is a strictly convex function
\begin{equation*}
\nabla^2_{(\xi,\eta)} \hat e > 0.
\end{equation*}
Comparing this definition with the one for the isothermal problem, it is evident that
the properties of weak continuity or regularity of cofactors and determinants
and their derivatives naturally follow from the isothermal regime. 
The polyconvexity assumption leads to reformulate system \eqref{e:adiabatic_thermoelasticity},~\eqref{constraint}
and write it instead in the variables $v\in\mathbb{R}^3,$ $\xi\in\mathbb{R}^{19}$ and $\eta\in\mathbb{R}$ 
regarding the vector $\Phi(F)=(F,\mathrm{cof}F,\det F)$ as a new independent variable. We perform this calculation
in detail in Section \ref{2:SEC_PRELIMINARIES}, and this is a simple variant of the extension in the variables
$(v,\xi,\theta),$ which can be found in \cite{CGT2017}, as $\theta$ and $\eta$ are connected via a Legendre transform.
For the latter, we refer the reader to \cite[Appendix B]{CGT2018}. The resulting augmented system is
\begin{align}
\label{intro:augmented_system}
\partial_t \xi^B-\partial_{\alpha}\left(\phv \right)&=0 \\ 
\partial_t v_i-\partial_{\alpha}\left(\frac{\partial \hat{e}}{\partial\xi^B}(\xi,\eta)\ph\right)&=0\\
\partial_t \eta &=\frac{r}{\hat{\theta}(\xi,\eta)}\\
\partial_t \left(\frac{1}{2}|v|^2+\hat{e}(\xi,\eta)\right)-\partial_{\alpha}\left(\frac{\partial\hat{e}}{\partial\xi^B}(\xi,\eta)\ph v_i\right)&= r
\qquad\qquad\qquad (B=1,\dots,19)
\end{align}
where $\hat{\theta} = \partial\hat{e}(\xi,\eta)/\partial\eta.$
System \eqref{intro:augmented_system} is symmetrizable and hyperbolic, and possesses a convex entropy.
In addition, \eqref{intro:augmented_system} preserves the null-Lagrangian structure and 
	the differential constraints \eqref{constraint} (see Section \ref{subsection 2}).

We work in the periodic domain $\mathbb{T}^3$ (in space) and construct a discrete in time, variational approximation scheme.
The scheme is implicit-explicit and decreases the energy. It depends on solving the following minimization problem: 
Given $h>0$ and initial data $U^0=(v^0,\xi^0,\eta^0)\in (L^2\times(L^p\times L^q\times L^{\rho})\times L^{\ell})(\mathbb{R}^{23}),$
there exists a unique $U=(v,\xi,\eta)\in (L^2\times(L^p\times L^q\times L^{\rho})\times L^{\ell})(\mathbb{R}^{23})$ 
that minimizes the functional
\begin{align}
\label{e:Functional}
J(v,\xi,\eta)=\int_{\mathbb{T}^3} \frac{1}{2}|v-v^0|^2+\hat{e}(\xi,\eta)\:dx
\end{align}
over the weakly closed affine subspace of $(v,\xi,\eta)\in L^2\times(L^p\times L^q\times L^{\rho})\times L^{\ell}$ such that
\begin{align}
\label{e:DISCRET_C}
\begin{split}
\frac{(\xi-\xi^0)^B}{h}&=\partial_{\alpha}\left(\frac{\partial\Phi^B}{\partial F_{i\alpha}}(F^0)v_i\right) \\ 
\frac{\eta-\eta^0}{h}&=\frac{r}{\hat{\theta}(\xi^0,\eta^0)}
\end{split}
\end{align}
hold in the sense of distributions. The minimizer satisfies the Euler-Lagrange equations
\begin{align}
\label{e:Euler-Lagrange}
\frac{(v-v^0)_i}{h}=\partial_{\alpha}\left(\frac{\partial \hat{e}}{\partial\xi^B}(\xi,\eta)\frac{\partial\Phi^B}{\partial F_{i\alpha}}(F^0)\right)
\end{align}
and the solution operator $S_h:U^0\mapsto U$ defined by the equations \eqref{e:DISCRET_C},~\eqref{e:Euler-Lagrange} preserves 
the constraints \eqref{constraint} which guarantee that at each time-step $j=0,1,2\dots$ the scheme produces iterations $(v^j,\xi^j,\eta^j)$ where $F^j$ 
is a deformation gradient. 
Moreover, these iterations satisfy the balance of entropy as an \textit{identity} while the energy equation
holds as an \textit{inequality}; those are relations \eqref{e:DISCRET_C} and \eqref{dicrete_energy_ineq} respectively. 

As a result, we solve a variant of the isentropic problem and the emerging solution dissipates the total energy.
On a mathematical side, we prove that the resulting variational approximation scheme gives
rise to measure-valued solutions that dissipate (instead of conserving) the energy identity, see Theorem \ref{mainconv}.
Our conclusion is similar in spirit with the findings of  \cite{Cavalletti} concerning the compressible Euler equations in Eulerian
coordinates. It is conceivable that using a different thermodynamic potential might lead to 
a variational scheme that balances the energy and increases the entropy, but this is at present an open problem.

We note that considering an energy inequality
\begin{align}
\label{e:adiabatic_thermoelasticity_energy}
\partial_t\left(\frac{1}{2}|v|^2+e\right)-\partial_{\alpha}(\Sigma_{i\alpha}v_i)\leq r
\end{align}
in lieu of an energy balance equation, it has attracted some attention in the mechanics literature. The balance of energy holds 
for thermodynamic processes that obey the first law of thermodynamics in a strong sense: For cyclic processes,
the work is universally proportional to heat; this is Joule's relation. As an alternative, a weaker form of Joule's law holding as an
inequality leads to an energy inequality rather than the balanced equation \cite{MR848765_Ser,MR589945,MR848765_Sil}. A side note on
the work of Fosdick and Serrin \cite{Fosdick_Serrin} yields that the assumption of a strong first law is not necessary for validating 
the classical statements of the second law of thermodynamics such as the Clausius-Duhem inequality and a weaker formulation, namely 
an energy inequality, suffices. We also refer the reader to  Serrin \cite{MR1420151}, where the author regards continuous
media as thermodynamical systems and derives the energy equation and the Clausius-Duhem inequality based solely on suitable
constitutive assumptions and the basic cyclic laws of thermodynamics; then assuming a weak first law, he proves the existence of 
the entropy and energy functions, validates the Clausius-Duhem inequality and derives an energy inequality.

We organize this paper as follows: In Section \ref{2:SEC_PRELIMINARIES}, we extend system \eqref{e:adiabatic_thermoelasticity}, \eqref{constraint} into a
symmetrizable, hyperbolic system, exploiting the polyconvex structure of the problem. In Section \ref{3:SEC3}, we give an outline
of the variational scheme and its main properties (for analogous studies of the isothermal problem see for instance 
\cite{MR1831179,MR3148623}, while other related work includes \cite{Cavalletti,Westdickenberg2017}).
We state and prove the minimization Theorem \ref{minimization} in Section \ref{3:SEC4}  and as a consequence of that, in Lemma \ref{dissipation}
we show that the scheme dissipates the energy which in turn, leads to the stability estimate \eqref{uniform_energy_estimate6}.
In Section \ref{3:SEC5}, we prove our main result, that the variational scheme generates a dissipative measure-valued solution
for isentropic processes in polyconvex thermoelasticity. Finally, in Section \ref{3:SEC6}, we show convergence of the solution generated by the
scheme, as the time step tends to zero, provided that the limit solution is smooth.

%	The goal is to show that the solutions constructed
%	via the variational scheme converge to the solution of \eqref{e:adiabatic_thermoelasticity} so long as the latter is smooth.
%	
%	in this same class of solutions, 
%	the measure-valued versus strong uniqueness result of \cite{CGT2018} continues to hold. We do this by means of a relative entropy calculation
%	for the extended system.

%NOTE: WE NEED TO CHANGE THIS PARAGRAPH IF WE DO THE EXTRA CHAPTER

%
%Section 2
%
%\vfil\eject

\section{Preliminaries}\label{2:SEC_PRELIMINARIES}

\subsection{Symmetrizable systems of conservation laws}
Systems of conservation laws describing the evolution of  a function $U : \mathbb{R}^d \times \mathbb{R}^+ \to \mathbb{R}^n$ have the form
\begin{align}
\label{syscl}
\partial_t A(U) +\partial_{\alpha}f_{\alpha}(U)=0
\end{align}
where $A, f_\alpha : \cO \subset  \mathbb{R}^n \to \mathbb{R}^n$, $\alpha =1, \dots, d$,  are smooth functions describing fluxes, 
$A(U)$ is globally  invertible (on the domain of definition $\cO \ni U$) and $\nabla A(U)$ is nonsingular.
The system \eqref{syscl} is endowed with an entropy - entropy flux pair $\eta, q_\alpha :   \mathbb{R}^n \to \mathbb{R}$ if any smooth solution
$U(x,t) \in C^1$ of \eqref{syscl}
satisfies the additional conservation law
\begin{align}
\label{entrpair}
\partial_t \eta(U) +\partial_{\alpha}q_{\alpha}(U) = 0 \, .
\end{align}

Existence of an entropy pair $\eta - q$ implies that there is a multiplier $G : \mathbb{R}^n \to \mathbb{R}^n$ such that $G = G(U)$ satisfies
\begin{equation}
\label{entropyG}
\begin{aligned}
G \cdot \nabla A &= \nabla \eta 
\\
G \cdot \nabla f_\alpha  &= \nabla q_\alpha .
\end{aligned}
\end{equation}
In turn, these are respectively equivalent to
\begin{equation}
\label{entropystruc}
\begin{aligned}
\nabla G^T  \nabla A  &= \nabla A^T \nabla G
\\
\nabla G^T  \nabla f_\alpha  &= {\nabla f_\alpha}^T \nabla G\,.
\end{aligned}
\end{equation}

Suppose now that \eqref{syscl} is endowed with a smooth entropy pair $\eta - q$, that is for some multiplier $G(U)$ 
relations \eqref{entropystruc} are satisfied. We rewrite \eqref{syscl} for smooth solutions, in the form of an equivalent system with
symmetric coefficients:
\begin{equation}
(\nabla G^T  \nabla A ) \del_t U + ( \nabla G^T  \nabla f_\alpha ) \del_\alpha U = 0.
\end{equation}
The hypothesis
$$
\nabla G^T  \nabla A  > 0
$$
guarantees that the system \eqref{syscl} is symmetrizable, it has real eigenvalues and is hyperbolic. Moreover, it induces a relative entropy identity
and a notion of stability for the system \cite{MR0285799,christoforou2016relative}. 
Using \eqref{entrpair},  it can be equivalently expressed in the form
\begin{equation}
\label{symm}
\nabla^2 \eta - \sum_{k=1} G^k  \nabla^2 A^k   > 0 \, .
\end{equation}
For weak solutions the entropy pair $\eta - q$ induces a notion of admissibility. The function $U \in L^1_{loc}$ is an entropy weak solution if
it satisfies \eqref{syscl}
and the inequality
\begin{align}
\label{entrsoln}
\partial_t \eta(U) +\partial_{\alpha}q_{\alpha}(U) \le 0
\end{align}
both in the sense of distributions. 

Weaker notions for solutions have been employed in the literature associated with
averaged forms of the system and the study of oscillations and concentrations. The notion of dissipative measure-valued 
solution is intended to describe the equations satisfied by weak limits emerging out of suitable $L^p$ bounds ($1<p<\infty$):
A dissipative measure-valued solution $(U,\bn,\bg)$ with concentration consists of a function
$U\in L^{\infty}([0,T];L^p(\mathbb{T}^d)),$ a Young measure $\bn=\bn_{(x,t)\in \mathbb{T}^d\times[0,T]}$ and a non-negative
Radon measure $\bg\in\mathbb{M}^+(\mathbb{T}^d\times[0,T])$ such that $U(x,t)=\left\la\bn,\lambda\right\ra$ while
$(\bn,\bg)$ satisfy
\begin{equation}
\label{weakform1}
\begin{aligned}
\int_0^T\int\left\la\bn,A_i(\lambda)\right\ra\partial_t\phi_i dx\:dt
+\int_0^T\int\left\la\bn, f_{\alpha i}(\lambda)\right\ra\partial_{\alpha}\phi_i dx\:dt
+\int\left\la\bn_0,A_i(\lambda)\right\ra \phi_i dx=0,
\end{aligned}
\end{equation}
for $\phi(x,t)\in C^1_c([0,T);C^1(\mathbb{T}^d;\mathbb{R}^n))$, and 
\begin{align}
\label{general_cl_dissipative_entropy}
\int_0^T\int\varphi^{\prime}\left(\left\la\bn,\eta(\lambda)\right\ra dx\:dt +\bg(dx\:dt) \right)
+\int\varphi(0)\left(\left\la\bn_0,\eta(\lambda)\right\ra dx +\bg_0(dx\:dt) \right)\geq 0
\end{align}
for $\varphi(t)\in C^1_c([0,T))$ with $\varphi\geq 0$. 
\tcr{Note that  in \eqref{weakform1} $f_{\alpha i}$ stands for the $i$-th coordinate of the vector function $f_{\alpha}$.}

The solution that we will construct for the system of  adiabatic thermoelasticity is analogous to the one just described, but an integrated averaged 
energy inequality with concentration  will substitute \eqref{general_cl_dissipative_entropy} while the averaged law of entropy will hold as an identity. 
This is due to structural properties of the problem and its relation to the  minimization scheme. Finally, we note that uniqueness of smooth solutions
in this class of dissipative measure-valued solutions for adiabatic polyconvex thermoelasticity was proven in \cite{CGT2018}.

\subsection{The augmented system of adiabatic polyconvex thermoelasticity}\label{subsection 2}

The system of adiabatic thermoelasticity consists of the equations \eqref{e:adiabatic_thermoelasticity}, \eqref{constraint} with the
constitutive relations \eqref{e:constitutive_functions}. Various thermodynamic potentials can be employed to determine thermodynamic theories (see \cite{Callen}); 
here, we focus on a theory with prime variables $(F, \eta)$ determined by the internal energy $e (F, \eta)$.

The system \eqref{e:adiabatic_thermoelasticity} fits into the general theory of conservation laws in two ways: 

One perspective proceeds following the continuum mechanics derivation of the theory of thermoelasticity. For smooth processes, $(y, \eta)(x,t) $, with 
$v= y_t$, $F = \nabla y$,  is viewed 
as a solution of the equations \eqref{e:adiabatic_thermoelasticity}$_{1,2,4}$, \eqref{constraint},  depicting conservation of momentum and energy. 
Then \eqref{e:adiabatic_thermoelasticity}$_{3}$ stands for the conservation of entropy. It is an additional conservation law, and a consequence of
the requirement of consistency of thermoelasticity with the Clausius-Duhem inequality \cite{CN1963}. Accordingly, when dealing with non-smooth processes 
$(y, \eta)(x,t) $, equation \eqref{e:adiabatic_thermoelasticity}$_{3}$ is replaced by an inequality 
\begin{equation}
\label{entropyeqn1}
\partial_t \eta  \ge  \frac{r}{\theta} 
\end{equation}
intended as an admissibility criterion for weak solutions which  motivates the concept of entropy
solution widespread in the theory of conservation laws.

To fit this perspective into the general form of system \eqref{syscl}, we set
$$
U = (F, v, \eta) \quad A(U) = \big ( F, v, \tfrac{1}{2} |v|^2 + e (F, \eta) \big )
$$
and note that the condition $\theta = \frac{\del e}{\del \eta} > 0$ guarantees that $A(U)$ is invertible and $\nabla A(U)$ is nonsingular.
By construction of the theory, there is a multiplier $G(U)$ that leads to the entropy pair $\check \eta (U)$-$\check q_\alpha (U)$ with
$$
\begin{aligned}
\check \eta (U) :=   - \eta,  \quad  \check q_\alpha (U) := 0, \quad
G(U) = \tfrac{1}{\frac{\del e}{\del \eta}(F,\eta)} \Big ( \frac{\del e}{\del F}(F,\eta)  , v , -1 \Big ) \, .
\end{aligned}
$$
A computation shows
\begin{equation}
\label{hyperbolic(F,eta)}
\nabla^2 \check \eta (U)  - \sum_{k=1} G^k(U)  \nabla^2 A^k(U) 
%\stackrel{\eqref{maxwell}}{=}
=
\frac{1}{\frac{\del e}{\del \eta} }
\begin{pmatrix}
e_{FF} & 0 & e_{F \eta}
\\
0 &  1 & 0
\\
e_{F \eta} & 0 & e_{\eta \eta}
\end{pmatrix}
\end{equation}
and thus the condition of symmetrizability \eqref{symm} amounts to $e(F, \eta)$ strictly convex and $\frac{\del e}{\del \eta} > 0$.
The requirement of convexity is too stringent to encompass a large class of materials and relaxing it, it is discussed below. Nevertheless,
it should be noted that convexity of $e(F,\eta)$ would suffice to apply the standard theory of conservation laws to \eqref{e:adiabatic_thermoelasticity}.
In that case the entropy admissibility inequality \eqref{entrsoln} would amount to the growth of the physical entropy \eqref{entropyeqn1}.

A second way to fit \eqref{e:adiabatic_thermoelasticity} to the general theory of conservation laws amounts
to view the process $(y, \eta)(x,t) $, through $v= y_t$, $F = \nabla y$, as solving \eqref{e:adiabatic_thermoelasticity}$_{1,2,3}$
and  \eqref{e:adiabatic_thermoelasticity}$_{4}$
as an additional conservation law. This is achieved by setting $A(U) = U = (F, v, \eta)$ and noting that the entropy - entropy flux pair
$$
\check \eta (U) = \tfrac{1}{2} |v|^2 + e(F,\eta) \quad \check q_\alpha (U) = - \Sigma_{i\alpha} v_i
$$
satisfies \eqref{e:adiabatic_thermoelasticity}$_{4}$.  Again convexity of $e(F,\eta)$ suffices to guarantee \eqref{symm} and apply the standard theory of conservation laws 
to \eqref{e:adiabatic_thermoelasticity}. In this case, the entropy inequality \eqref{entrsoln} would imply that a weak solution satisfies
the energy inequality
\begin{equation}
\label{entropyeqn2}
\partial_t\left(\frac{1}{2}|v|^2+e (F, \eta) \right)-\partial_{\alpha}(\Sigma_{i\alpha}v_i) \le r.
\end{equation}
An energy inequality contravenes the traditional view of the first law of thermodynamics, However,  in the mechanics literature \cite{Fosdick_Serrin,MR1420151,MR589945,MR848765_Ser,MR848765_Sil} the role of an energy inequality in the derivation of thermodynamics has been extensively studied, and it was notably established
by  Serrin \cite{MR1420151} that, for a wide class of constitutive relations,  postulating a weak form of the first law, it
still leads to the existence of an energy and entropy function and is consistent with the second law in the form
of the Clausius-Duhem inequality.

Next, we turn to the assumptions on the internal energy $e(F, \eta)$.
The convexity of internal energy in $\eta$ is intrinsic in the derivation of thermodynamics, see \cite{CN1959}, \cite[Ch I]{Evans}.
If we impose that $e(F,\eta)$ is convex and coercive for $\eta>0:$
\begin{equation*}
\frac{\del^2 e}{\del \eta^2}  (F, \eta) > 0 ,  \qquad \lim_{\eta \to \infty}  \frac{ e(F, \eta)}{\eta} = \infty
\end{equation*}
and that the temperature is zero when the entropy is zero, namely $\theta(F, 0)=0$, then we have the implication
\begin{equation*}
\theta(F,\eta)=\frac{\partial e(F,\eta)}{\partial\eta}>0.
\end{equation*}

By contrast,  for thermoelastic materials convexity of the stored energy with respect to $F$ 
together with the requirement of frame indifference are in general incompatible with the hypothesis that the stored energy becomes infinite in the limit as $\det F \to 0$,
which is in turn necessary to avoid interpenetration of matter \cite{MR3468916,TN1992,CN1963}. 
To relax the requirement of convexity, the hypotheses of  polyconvexity, quasiconvexity or rank-1 convexity are often employed; 
here, we take up the hypothesis of polyconvexity introduced by Ball \cite{MR0475169} in the theory of elasticity and connected to the notion of null-Lagrangians \cite{BCO81}.

The hypothesis of polyconvexity is quite useful in dynamic elasticity  \cite{MR1651340,MR1831179,MR3468916,Wagner2009,LT2006,MR1831175,MR3232713,Duan2017}),
as it leads to embedding the (isothermal) elasticity system to an augmented symmetric, hyperbolic system. It was recently realized that the system of adiabatic
thermoelasticity can be extended into an augmented symmetric hyperbolic system \cite{CGT2017}, and this property is employed here
to construct a variational approximation scheme for \eqref{e:adiabatic_thermoelasticity}.

To properly formulate the problem, we introduce a variant of polyconvexity, along the lines of  \cite{CGT2017,CGT2018}, 
according to which the free energy $e(F,\eta)$ factorizes 	
\begin{equation}
\label{e:polyconvex_fe}
e (F, \eta) =  \hat{e}(\Phi(F),\eta) \, ,
\end{equation}
where 
$$
\Phi(F)=(F,\mathrm{cof}F,\det F)  \; \in \, \mathbb{M}^{3\times3}\times\mathbb{M}^{3\times3}\times\mathbb{R}(\simeq\mathbb{R}^{19})
$$
is the vector of null-Lagrangians. We require $\hat{e}$ to be strictly convex with respect to the variables $(\xi, \eta) \in\mathbb{R}^{19}\times\mathbb{R}$,
that is 
\begin{equation}
\label{e:conditions_polyconvex}
\nabla^2_{(\xi,\eta)} \hat e > 0.
\end{equation}
We call the assumptions \eqref{e:polyconvex_fe}-\eqref{e:conditions_polyconvex} polyconvexity in the non-isothermal context.

Taking advantage of the null-Lagrangian structure and  the polyconvexity condition, system \eqref{e:adiabatic_thermoelasticity} is embedded
into an augmented symmetrizable system, see \cite{CGT2017}, as follows:
For $d=3,$  the cofactor matrix $\mathrm{cof}F\in\mathbb{M}^{3\times3}$ and 
determinant $\det F\in\mathbb{R}$ of $F$ are defined as
\begin{align}
\label{cofactor_d=3}
(\mathrm{cof}F)_{i\alpha} &=
\frac{1}{2}\epsilon_{ijk}\epsilon_{\alpha\beta\gamma}
F_{j\beta}F_{k\gamma},\\
\label{det_d=3}
\dt F &=\frac{1}{6}\epsilon_{ijk}\epsilon_{\alpha\beta\gamma}
F_{i\alpha}F_{j\beta}F_{k\gamma} = \frac{1}{3}(\mathrm{cof}F)_{i \alpha}F_{i \alpha}.
\end{align}
The vector of null-Lagrangians $\Phi^B(F)$, $B = 1,\dots,19$, satisfies the Euler-Lagrange equations
\begin{equation}
\label{Euler-Lagrange}
\partial_{\alpha}\left(\frac{\del \Phi^B}{\del F_{i\alpha}} (\nabla y ) \right)=0. %\, , \quad B=1,\dots,19.
\end{equation}
Combining \eqref{Euler-Lagrange} with \eqref{e:adiabatic_thermoelasticity}$_1$ allows us to write the equation
\begin{align}
\label{dtPhitotal}
\partial_t \Phi^B(F)=\partial_{\alpha}\left(\phv\right),
\end{align}
which induces two additional conservation laws (alongside with \eqref{e:adiabatic_thermoelasticity}$_1$) (cf. \cite{MR1651340})
\begin{align}
\label{transportQIN}
\begin{split}
\partial_t\det F&=\partial_{\alpha}\bigl((\mathrm{cof} F)_{i\alpha}v_i\bigr)\\
\partial_t(\mathrm{cof} F)_{k\gamma}&=\partial_{\alpha}(\epsilon_{ijk}\epsilon_{\alpha\beta\gamma}F_{j\beta}v_i).
\end{split}
\end{align}
Observe that because of \eqref{e:polyconvex_fe}, the stress tensor $\Sigma$, given by \eqref{e:constitutive_functions}$_2,$
can be written with respect to $\Phi(F):$
\begin{align}
\label{sigma-Phi}
\Sigma_{i\alpha}=\frac{\partial e}{\partial F_{i\alpha}}(F,\eta)=\frac{\partial }{\partial F_{i\alpha}}\left(\hat{e}(\Phi(F),\eta)\right)
=\frac{\partial \hat{e}}{\partial\xi^B}(\Phi(F),\eta)\ph.
\end{align}
The same holds for the temperature function $\theta$ since it can also be expressed as function of $\hat{e}$ 
because the first nine components of $\Phi(F)$ are the components of the matrix $F.$  We then have
\begin{align}
\label{theta-Phi}
\theta(F,\eta)=\frac{\partial e}{\partial\eta}(F,\eta)=\frac{\partial\hat{e}}{\partial\eta}(\Phi(F),\eta)=\hat{\theta}(\Phi(F),\eta),
\quad\text{where}\quad\hat{\theta}(\xi,\eta):=\frac{\partial\hat{e}}{\partial\eta}(\xi,\eta).
\end{align}

By virtue of the identities  \eqref{sigma-Phi} and  \eqref{theta-Phi}, the system of adiabatic thermoelasticity \eqref{e:adiabatic_thermoelasticity}-\eqref{e:constitutive_functions}
can be expressed for smooth solutions in the form
\begin{align}
\label{e:adiabatic_thermoelasticity_Phi}
\begin{split}
\partial_t \Phi^B(F)-\partial_{\alpha}\left(\phv \right)&=0 \\ 
\partial_t v_i-\partial_{\alpha}\left(\frac{\partial \hat{e}}{\partial\xi^B}(\Phi(F),\eta)\ph\right)&=0\\
\partial_t \eta &=\frac{r}{\hat{\theta}(\Phi(F),\eta)}\\
\partial_t \left(\frac{1}{2}|v|^2+\hat{e}(\Phi(F),\eta)\right)-\partial_{\alpha}\left(\frac{\partial\hat{e}}{\partial\xi^B}(\Phi(F),\eta)\ph v_i\right)&= r 
\end{split}
\end{align}
subject to the constraint \eqref{constraint}.

The system  \eqref{e:adiabatic_thermoelasticity_Phi} may be written in terms of the extended variable
$\xi:=(F,\zeta,w)\in\mathbb{M}^{3\times3}\times\mathbb{M}^{3\times3}\times\mathbb{R}(\simeq\mathbb{R}^{19})$
in the following manner
\begin{align}
\label{e:adiabatic_thermoelasticity_augmented}
\begin{split}
\partial_t \xi^B-\partial_{\alpha}\left(\phv \right)&=0 \\ 
\partial_t v_i-\partial_{\alpha}\left(\frac{\partial \hat{e}}{\partial\xi^B}(\xi,\eta)\ph\right)&=0\\
\partial_t \eta &=\frac{r}{\hat{\theta}(\xi,\eta)}\\
\partial_t \left(\frac{1}{2}|v|^2+\hat{e}(\xi,\eta)\right)-\partial_{\alpha}\left(\frac{\partial\hat{e}}{\partial\xi^B}(\xi,\eta)\ph v_i\right)&= r \, 
\end{split}
\end{align}
subject again to the constraint \eqref{constraint}.
The functions $\hat{e}$, $\hat{\theta}$ and $\hat{\Sigma}$ are connected through the formulas
\begin{equation}
\label{ext-constitutive}
\hat{\Sigma}^B (\xi, \eta) := \frac{\partial \hat{e}}{\partial\xi^B} (\xi, \eta) \, , \quad  \hat{\theta}(\xi,\eta):=\frac{\partial\hat{e}}{\partial\eta}(\xi,\eta).
\end{equation}
A computation as in \eqref{hyperbolic(F,eta)} shows that if $\hat{e}(\xi,\eta)$ is strictly convex, 
the system of conservation laws \eqref{e:adiabatic_thermoelasticity_augmented}$_{1,2,4}$ is symmetrizable
and hyperbolic. This enlarged system has the following properties \cite{CGT2017}:
\begin{itemize}
	\item[(i)]  The extension preserves the constraint \eqref{constraint}: if $F(\cdot,t=0)$ is a deformation gradient, then $F(\cdot,t)$ remains a gradient for all times. 
	\item[(ii)] In addition,
	$$
	\xi(\cdot,t=0)=\Phi(F(\cdot,t=0)) \quad \mbox{implies} \quad \xi(\cdot,t)=\Phi(F(\cdot,t)) \; \; \forall t.
	$$
\end{itemize}
In other words, we can regard \eqref{e:adiabatic_thermoelasticity} as a constrained evolution of the 
augmented problem \eqref{e:adiabatic_thermoelasticity_augmented}. 
Furthermore, system \eqref{e:adiabatic_thermoelasticity_augmented}$_{1,2,4}$ is endowed with the additional conservation law
of entropy  \eqref{e:adiabatic_thermoelasticity_augmented}$_{3}$ for smooth solutions. To see this, we
multiply \eqref{e:adiabatic_thermoelasticity_augmented}$_{1,2,4}$ by
$$ G(U) = \tfrac{1}{\hat{\theta}(\xi,\eta)}\Big( \hat{\Sigma}^B (\xi, \eta), v, -1\Big )$$
and use \eqref{ext-constitutive} and \eqref{Euler-Lagrange} to obtain \eqref{e:adiabatic_thermoelasticity_augmented}$_{3}$.

\section{The discrete scheme}\label{3:SEC3}

We consider a time-discretized variant of the system \eqref{e:adiabatic_thermoelasticity_augmented}:
Given $(v^0,\xi^0,\eta^0)$ and $h>0,$ consider the discrete equations
\begin{align}
\label{discrete_thermoelasticity}
\begin{split}
\frac{(\xi-\xi^0)^B}{h}&=\partial_{\alpha}\left(\phz v_i\right)\\
\frac{v_i-v_i^0}{h}&=\partial_{\alpha}\left(\pge\phz\right)\\
\frac{\eta-\eta^0}{h}&=\frac{r}{\hat{\theta}(\xi^0,\eta^0)}.
\end{split}
\end{align}

We next give an outline of the main properties of the scheme, avoiding technical details that will be addressed
in the following Section. The iterates $(v, \xi, \eta)$ are constructed as the unique solutions
of the minimization problem
\begin{align}
\label{minproblem}
\underset{ \scriptsize
	\begin{aligned}
	\frac{(\xi-\xi^0)^B}{h} &= \partial_{\alpha}\left(\phz v_i\right)  
	\\
	\frac{\eta-\eta^0}{h} &= \frac{r}{\hat{\theta}(\xi^0,\eta^0)} 
	\end{aligned} 
}{{\rm min} } 
\Big(  \int \frac{1}{2}|v-v^0|^2+\hat{e}(\xi,\eta)\:dx \Big ).
\end{align}
Note that the constraints are affine and are understood in the sense of distributions, while the functional is convex due to  \eqref{e:conditions_polyconvex}.
	The equations \eqref{discrete_thermoelasticity}$_2$ will turn out to be the Euler-Lagrange equations for the minimization problem 
	\eqref{minproblem}.

The second property is that the iterates satisfy a discrete energy inequality
\begin{equation}
\label{dicrete_energy_ineq}
\tfrac{1}{h}\left(\frac{1}{2}|v|^2+\hat{e}(\xi,\eta)-\frac{1}{2}|v^0|^2-\hat{e}(\xi^0,\eta^0)\right)
\leq \partial_{\alpha}\left(\pge\phz v_i\right)+\frac{\hat{\theta}(\xi,\eta)}{\hat{\theta}(\xi^0,\eta^0)} \, r .
\end{equation}
To deduce this property we use \eqref{e:conditions_polyconvex}.
We give a formal derivation of \eqref{dicrete_energy_ineq} here and defer the proof  to Lemma \ref{lemmaenergydiss}. Set
\begin{align}
\label{I(U)}
I(v,\xi,\eta):=\frac{1}{2}|v|^2+\hat{e}(\xi,\eta)
\end{align}
and define the relative (total) energy to be the quadratic part of the Taylor series expansion
\begin{align}
\label{relative_I(U)}
\begin{split}
I(v^0, &\xi^0,\eta^0|v,\xi,\eta)
:=I(v^0,\xi^0,\eta^0)-I(v,\xi,\eta)-I_v(v,\xi,\eta)\cdot(v^0-v)\\
&\qquad\qquad\qquad\qquad\quad-I_{\xi}(v,\xi,\eta)\cdot(\xi^0-\xi)-I_{\eta}(v,\xi,\eta)(\eta^0-\eta)\\
&=\frac{1}{2}|v^0 - v|^2+\hat{e}(\xi^0,\eta^0) -\hat{e}(\xi,\eta) -\pge(\xi^0-\xi)^B-\frac{\partial\hat{e}}{\partial\eta}(\xi,\eta)(\eta^0-\eta) 
\\
&= \frac{1}{2}|v^0 - v|^2 + \hat{e}(\xi^0,\eta^0 |  \xi, \eta ) 
\end{split}
\end{align} 
Next, write
\begin{align*}
\frac{1}{2}|v|^2+\hat{e}(\xi&,\eta)-\frac{1}{2}|v^0|^2-\hat{e}(\xi^0,\eta^0)=\\
&=-I(v^0,\xi^0,\eta^0|v,\xi,\eta)+v_i(v_i-v_i^0)+\pge(\xi-\xi^0)^B
+\frac{\partial\hat{e}}{\partial\eta}(\xi,\eta)(\eta-\eta^0) \, ,
\end{align*} 
and using  \eqref{discrete_thermoelasticity}, \eqref{ext-constitutive} 
and the null-Lagrangian identity \eqref{Euler-Lagrange} the latter implies
\begin{align}
\label{relative_energy_identity}
\begin{split}
\frac{1}{h}\left(\frac{1}{2}|v|^2+\hat{e}(\xi,\eta)
-\frac{1}{2}|v^0|^2\right.&\left.\vphantom{\frac{1}{2}}-\hat{e}(\xi^0,\eta^0)\right)+\frac{1}{h}I(v^0,\xi^0,\eta^0|v,\xi,\eta)=\\
&=\partial_{\alpha}\left(\pge\phz v_i\right)
+\hat{\theta}(\xi,\eta)\frac{r}{\hat{\theta}(\xi^0,\eta^0)}.
\end{split}
\end{align} 
Under the convexity assumption on $\hat{e}(\xi,\eta),$ we infer that $I(v^0,\xi^0,\eta^0|v,\xi,\eta)\geq 0,$
and \eqref{relative_energy_identity} yields \eqref{dicrete_energy_ineq}.

Given a time step of size $h>0$, we split the interval $[0,T]$, $T = N h$, into subintervals of length $h$ and construct the iterates at
the time nodes $j h$, $j= 0, 1, ..., N$, by solving a time-discretized version of the system \eqref{e:adiabatic_thermoelasticity_augmented}
in the extended variables. For $U^0=(v^0,\xi^0,\eta^0)\in (L^2\times(L^p\times L^q\times L^{\rho})\times L^{\ell})(\mathbb{R}^{23})$
we solve the minimization problem \eqref{minproblem}. The procedure is carried out in Theorem \ref{minimization};  the minimizer solves
equations \eqref{discrete_thermoelasticity} in the sense of distributions. The iterates satisfy \eqref{dicrete_energy_ineq} again in the sense of distributions, 
see corollary \ref{minimizer_energy}. 

Given initial data 
$(F^0,v^0,\eta^0)^T=(F(x,0),v(x,0),\eta(x,0))^T,$ this procedure defines a solution operator $S_h:\mathbb{R}^{23}\to\mathbb{R}^{23}$ 
determined by equations \eqref{discrete_thermoelasticity}
$$
S_h:(v^{j-1},F^{j-1},\zeta^{j-1},w^{j-1},\eta^{j-1})\mapsto(v^j,F^j,\zeta^j,w^j,\eta^j) \, .
$$
A notable feature of the scheme is that it decreases the energy, this is Lemma \ref{dissipation},
and that in turn provides a uniform bound on the iterates thus rendering the scheme stable. We construct those
iterates as follows: At the $j$-th step, the iterate is
\begin{align*}
(v^j,\xi^j,\eta^j)=(v^j,F^j,\zeta^j,w^j,\eta^j)&=S_h(v^{j-1},F^{j-1},\zeta^{j-1},w^{j-1},\eta^{j-1})=S_h(v^{j-1},\xi^{j-1},\eta^{j-1})\\
&=S_h^j(v^0,F^0,\zeta^0,w^0,\eta^0)=S_h^j(v^0,\xi^0,\eta^0)
\end{align*}
so that $U^j=S_h^j(U^0).$ Here, we denote by $S_h^j$ the composition of $S_h$ by itself $j$-times. Given the $(j-1)$-th iterative solution $(v^{j-1},\xi^{j-1},\eta^{j-1})$ 
the solution at the next step is determined by
\begin{align}
\label{discrete_thermoelasticity_j}
\begin{split}
\frac{(\xi^j-\xi^{j-1})^B}{h}&=\partial_{\alpha}\left(\phzj v_i^j\right)\\
\frac{v_i^j-v_i^{j-1}}{h}&=\partial_{\alpha}\left(\pgej\phzj\right)\\
\frac{\eta^j-\eta^{j-1}}{h}&=\frac{r}{\hat{\theta}(\xi^{j-1},\eta^{j-1})}.
\end{split}
\end{align}

Next we list the technical hypotheses required for the analysis. 
Throughout the text, we assume $\hat{e}\in C^3(\mathbb{R}^{19}\times[0,\infty))$ \tcb{is strictly convex} and impose the growth conditions 
\begin{equation}
\label{growth.con.1}
c(|F|^p+|\zeta|^q+|w|^{\rho}+|\eta|^{\ell})-c\leq \hat{e}(F,\zeta,w,\eta)\leq c(|F|^p+|\zeta|^q+|w|^{\rho}+|\eta|^{\ell})+c\;,
\end{equation}
\begin{equation}
\label{growth.con.2}
\lim_{|\xi|_{p,q,\rho}+|\eta|^{\ell}\to\infty}\frac{\partial_{\eta}\hat{e}(\xi,\eta)}{|\xi|_{p,q,\rho}+|\eta|^{\ell}}=
\lim_{|\xi|_{p,q,\rho}+|\eta|^{\ell}\to\infty}\frac{\hat{\theta}(\xi,\eta)}{|\xi|_{p,q,\rho}+|\eta|^{\ell}}=0\;,
\end{equation}
and
\begin{equation}
\label{growth.con.3}
|\partial_F\hat{e}|^{\frac{p}{p-1}}+|\partial_{\zeta}\hat{e}|^{\frac{p}{p-2}}+|\partial_w\hat{e}|^{\frac{p}{p-3}}\leq c(|F|^p+|\zeta|^q+|w|^{\rho}+|\eta|^{\ell})+c,
\end{equation}
for some constant $c>0$ and for $p\geq 4,$ $q \ge 2$, $\rho > 1$ and $\ell>1,$ where we used the notation:
\begin{equation*}
|\xi|_{p,q,\rho}:=|F|^p+|\zeta|^q+|w|^{\rho}.
\end{equation*}
\tcb{
	An example of a strictly convex function satisfying the above growth conditions is
	$$
	\hat{e}(\xi,\eta) = |F|^4+|F|^2+ g (\zeta) + h (w) + k(\eta)
	$$
	where $g(\zeta )$, $h(w)$ and $k(\eta)$ are smooth uniformly convex functions satisfying,  for $q \ge 2$ and $\rho, \ell > 1$,
	$$
	\begin{aligned}
	g(\zeta) &= |\zeta|^q  + o (|\zeta|^q)  \quad \mbox{as $|\zeta| \to \infty$}\, , 
	\\
	\quad h(w) &= |w|^\rho  + o (|w|^\rho)  \quad \mbox{as $|w| \to \infty$}\, , 
	\\
	\quad k(\eta)  &= |\eta|^\ell  + o (|\eta|^\ell)  \quad \mbox{as $|\eta| \to \infty$}\, .
	\end{aligned}
	$$
}

\section{The minimization problem}\label{3:SEC4}
In this Section, we prove that the discretization scheme \eqref{discrete_thermoelasticity}-\eqref{dicrete_energy_ineq} can be solved for all $h>0$
by a constrained minimization method which decreases the energy.
\begin{theorem}
	\label{minimization}
	Suppose the initial data
	\begin{align}
	\begin{split}
	\label{initial data}
	y^0\in W^{1,p}(\mathbb{T}^3), \quad \partial_ty^0=v^0\in L^2(\mathbb{T}^3), \quad \eta^0\in L^{\ell}(\mathbb{T}^3), \quad \text{such that if}\\
	F^0=\nabla y^0 \in L^p(\mathbb{T}^3) \quad \text{then} \quad (F^0, \cof F^0, \det F^0)\in (L^p\times L^q\times L^{\rho})(\mathbb{T}^3)
	\end{split}
	\end{align}
	and assume the growth condition \eqref{growth.con.1} on $\hat{e}(\xi,\eta).$ If $\hat{e}(\xi,\eta)$ is 
	strictly convex, there exists a unique
	$$(v,\xi,\eta)\in(L^2\times L_{p,q,\rho}\times L^{\ell})(\mathbb{T}^3)$$
	which minimizes the functional
	\begin{align}
	\label{functional}
	J(v,\xi,\eta)=\int\frac{1}{2}|v-v^0|^2+\hat{e}(\xi,\eta)\:dx
	\end{align}
	over the weakly closed affine subspace $\mathcal{C}$ of $(L^2\times L_{p,q,\rho}\times L^{\ell})\ni(v,\xi,\eta)$ defined as
	\begin{equation}
	\label{lemma_weak_formulation}
	\begin{aligned}
	\int\phi\frac{1}{h}(\xi-\xi^0)^B\:dx&=-\int\phz v_i\partial_{\alpha}\phi\:dx,\\
	\int\phi\frac{1}{h}(\eta-\eta^0)\:dx&=\int\frac{r}{\hat{\theta}(\xi^0,\eta^0)}\phi\:dx,
	\end{aligned}
	\qquad \phi\in C^{\infty}(\mathbb{T}^3).
	\end{equation}
	Assuming further \eqref{growth.con.3} the minimizer satisfies the Euler-Lagrange
	equation
	\begin{equation}
	\label{E-L_min}
	\int\phi\frac{1}{h}(v_i-v_i^0)\:dx=-\int\pge\phz\partial_{\alpha}\phi\:dx.
	\end{equation}
	Moreover, the condition \eqref{constraint} is preserved by the solution operator $S_h,$ which means that whenever $F^0$ is a gradient, then $F$ is also a gradient.
	Therefore there exists a function 
	\begin{align*}
	y:\mathbb{T}^3\to\mathbb{R}^3 \:\:\:\text{in}\:\:W^{1,p}(\mathbb{T}^3)
	\end{align*}
	such that $\partial_{\alpha}y_i=F_{i\alpha}.$
\end{theorem}
\begin{proof}
	\textbf{Step 1:} (Existence of a minimizer). Let $\inf\limits_{\mathcal{C}}J(v,\xi,\eta)=:m$ and observe that by virtue of \eqref{growth.con.1} there holds
	\begin{align*}
	J(v^0,\xi^0,\eta^0)=\int\hat{e}(\xi^0,\eta^0)\:dx
	\leq c\left(\|F^0\|^p_{L^p}+\|\zeta^0\|^q_{L^q}+\|w^0\|^{\rho}_{L^{\rho}}+\|\eta^0\|^{\ell}_{L^{\ell}}+|\mathbb{T}^3|\right)<\infty,
	\end{align*}
	so that $m<+\infty.$ Also $m>-\infty$ because similarly the lower bound in \eqref{growth.con.1} implies that
	\begin{align*}
	J(v,\xi,\eta)\geq\frac{1}{2}\|v-v^0\|^2_{L^2}+c\left(\|F\|^p_{L^p}+\|\zeta\|^q_{L^q}+\|w\|^{\rho}_{L^{\rho}}+\|\eta\|^{\ell}_{L^{\ell}}-|\mathbb{T}^3|\right)>-\infty.
	\end{align*}
	Let $(\ve,\xe,\ee)$ be a minimizing sequence: $J(\ve,\xe,\ee)\to m$ then indeed for $\varepsilon$ large enough
	\begin{align*}
	m+1\geq J(\ve,\xe,\ee)\geq c\left(\|\ve-v^0\|^2_{L^2}+\|\fe\|^p_{L^p}+\|\ze\|^q_{L^q}+\|\we\|^{\rho}_{L^{\rho}}+\|\ee\|^{\ell}_{L^{\ell}}\right)-c|\mathbb{T}^3|,
	\end{align*}
	therefore up to a subsequence 
	\begin{align*}
	\ve\rightharpoonup v\:\:\:\text{in}\:\:L^2(\mathbb{T}^3),\qquad
	\xe\rightharpoonup \xi\:\:\:\text{in}\:\:L_{p,q,\rho}(\mathbb{T}^3),\qquad
	\ee\rightharpoonup \eta\:\:\:\text{in}\:\:L^{\ell}(\mathbb{T}^3) \, .
	\end{align*}
	\tcb{
		Since $|\ve-\cdot|^2$ and $\hat{e}(\xe,\ee)$ are convex,
		%		\begin{align*}
		%		J(\ve,\xe,\ee)&=\int\frac{1}{2}|\ve-v^0|^2+\hat{e}(\xe,\ee)\:dx=\int\frac{1}{2}|\ve-v^0|^2+\hat{e}\left(\xe,\eta^0+h\frac{r}{\hat{\theta}(\xi^0,\eta^0)}\right)\:dx\\
		%		&=\int\frac{1}{2}|\ve-v^0|^2+\hat{e}(\xe,\bar{\eta})\:dx=:\tilde{J}(\ve,\xe),
		%		\end{align*}
		%		where $\ds\bar{\eta}:=\eta^0+h\frac{r}{\hat{\theta}(\xi^0,\eta^0)}.$ Then
		\begin{align*}
		\liminf_{\varepsilon}J(\ve,\xe,\ee) \geq\int\frac{1}{2}|v-v^0|^2+\hat{e}(\xi,\bar{\eta})\:dx = J(v,\xi,\eta) \, .
		\end{align*}
		Finally, as the constraints are affine, the weak limit $(v,\xi,\eta) \in \mathcal{C}$ and minimizes $J(v,\xi,\eta)$ over $\mathcal{C}.$
	}
	
	\medskip
	\textbf{Step 2:} (Uniqueness). Let $\bar{U}=(\bar{v},\bar{\xi},\bar{\eta})$ and 
	$\tilde{U}=(\tilde{v},\tilde{\xi},\tilde{\eta})\:\in\mathcal{C}$ such that $J(\bar{U})=J(\tilde{U})=m$ and set
	\begin{align*}
	U=\frac{\bar{U}+\tilde{U}}{2}=\left(\frac{\bar{v}+\tilde{v}}{2},\frac{\bar{\xi}+\tilde{\xi}}{2},\frac{\bar{\eta}+\tilde{\eta}}{2}\right)
	=\left(\frac{\bar{v}+\tilde{v}}{2},\frac{\bar{\xi}+\tilde{\xi}}{2},\eta^0+h\frac{r}{\hat{\theta}(\xi^0,\eta^0)}\right)\in\mathcal{C}.
	\end{align*}
	Then $U$ is also a minimizer:
	\begin{align*}
	&m \leq J\left(\frac{\bar{v}+\tilde{v}}{2},\frac{\bar{\xi}+\tilde{\xi}}{2},\frac{\bar{\eta}+\tilde{\eta}}{2}\right)
	=\int\left|\frac{\bar{v}+\tilde{v}}{2}-v^0\right|^2+\hat{e}\left(\frac{\bar{\xi}+\tilde{\xi}}{2},\eta^0+h\frac{r}{\hat{\theta}(\xi^0,\eta^0)}\right)\:dx\\
	&\leq\frac{1}{2}\int|\bar{v}-v^0|^2+\hat{e}\left(\bar{\xi},\eta^0+h\frac{r}{\hat{\theta}(\xi^0,\eta^0)}\right)\:dx
	+\frac{1}{2}\int|\tilde{v}-v^0|^2+\hat{e}\left(\bar{\xi},\eta^0+h\frac{r}{\hat{\theta}(\xi^0,\eta^0)}\right)\:dx\\
	&=\frac{1}{2}J(\bar{U})+\frac{1}{2}J(\tilde{U})=m,
	\end{align*}
	since $\xi\mapsto\hat{e}(\xi,\cdot)$ is convex. This implies
	\begin{align*}
	\int\left[\frac{1}{2}|\bar{v}-v^0|^2+\hat{e}\left(\bar{\xi},\eta^0+h\frac{r}{\hat{\theta}(\xi^0,\eta^0)}\right)
	+\frac{1}{2}|\tilde{v}-v^0|^2+\hat{e}\left(\bar{\xi},\eta^0+h\frac{r}{\hat{\theta}(\xi^0,\eta^0)}\right)\right.\\
	\left.-\left|\frac{\bar{v}+\tilde{v}}{2}-v^0\right|^2+\hat{e}\left(\frac{\bar{\xi}-\tilde{\xi}}{2},\eta^0+h\frac{r}{\hat{\theta}(\xi^0,\eta^0)}\right)\right]\:dx=0
	\end{align*}
	but the integrand is nonnegative since $(v,\xi)\mapsto|v-v^0|^2+\hat{e}(\xi,\cdot)$ is convex, thus it has to be identically zero. Since
	$\hat{e}$ is strictly convex we get $\bar{v}=\tilde{v}$ and
	$\bar{\xi}=\tilde{\xi}$ \textit{a.e.} so that $\bar{U}=\tilde{U}$ \textit{a.e.}.
	
	\medskip
	\textbf{Step 3:} (The Euler-Lagrange equation). Employing the notation $\delta g$ to indicate the variation of the function $g$,
	we consider a smooth variation generated by the functions $\phi_i:\mathbb{T}^3\to\mathbb{R}^3$ $(i=1,2,3):$
	\begin{align}
	\label{variation}
	(\epsilon\delta V_i,\epsilon\delta \Xi^B,\epsilon\delta H)=\left(\epsilon\phi_i,\epsilon h\phz\partial_{\alpha}\phi_i,0\right),\qquad(\epsilon>0)
	\end{align}
	by virtue of \eqref{lemma_weak_formulation}. Since $(v,\xi,\eta)$ is a minimizer
	\begin{align*}
	J(v,\xi,\eta)\leq J(v+\epsilon\delta V,\xi+\epsilon\delta\Xi,\eta+\epsilon\delta H)
	=J\left(v+\epsilon\phi,\epsilon h\frac{\partial\Phi}{\partial F_{i\alpha}}(F^0)\partial_{\alpha}\phi,\eta\right).
	\end{align*}
	Let us define
	\begin{align*}
	\int J^{\epsilon}(v,\xi,\eta)\:dx
	&:=\int\frac{1}{\epsilon}\left(J(v+\epsilon\delta V,\xi+\epsilon\delta\Xi,\eta+\epsilon\delta H)-J(v,\xi,\eta)\right)\:dx\\
	&=
	\int\frac{1}{\epsilon}\int_{0}^{1}\frac{d}{ds}
	J(v+s\epsilon\delta V,\xi+s\epsilon\delta\Xi,\eta+s\epsilon\delta H)\:dsdx
	\end{align*}
	Then using \eqref{variation}, the integrand becomes
		\begin{align*}
		J^{\epsilon} &(v,\xi,\eta) =
		\frac{1}{\epsilon}\int_{0}^{1}\frac{d}{ds}J(v+s\epsilon\delta V,\xi+s\epsilon\delta\Xi,\eta+s\epsilon\delta H)\:ds\\
		%		&=
		%		\frac{1}{\epsilon}\int_{0}^{1}\frac{d}{ds}
		%		\left(\frac{1}{2}|v+s\epsilon\delta V-v^0|^2+\hat{e}(\xi+s\epsilon\delta\Xi,\eta+s\epsilon\delta H)\right)\:ds\\
		&=
		\frac{1}{\epsilon}\int_{0}^{1}\left[(v+s\epsilon\delta V-v^0)_i \epsilon\delta V_i
		+\frac{\partial \hat{e}}{\partial \xi^B}(\xi+s\epsilon\delta\Xi,\eta+s\epsilon\delta H)\left(\epsilon h\phz\partial_{\alpha}\delta V_i\right)\right.\\
		&\left.\qquad\qquad\qquad\qquad\qquad\qquad\qquad\qquad\qquad
		+\frac{\partial \hat{e}}{\partial \eta}(\xi+s\epsilon\delta\Xi,\eta+s\epsilon\delta H)\epsilon\delta H\right]\:ds\\
		&=
		\int_{0}^{1}\left[(v+s\epsilon\delta V-v^0)_i \delta V_i
		+\frac{\partial \hat{e}}{\partial \xi^B}(\xi+s\epsilon\delta\Xi,\eta+s\epsilon\delta H)\left(h\phz\partial_{\alpha}\delta V_i\right)\right]\:ds\\
		&=\int_{0}^{1}\left[(v+s\epsilon\delta V-v^0)_i\phi_i+\frac{\partial \hat{e}}{\partial \xi^B}(\xi+s\epsilon\delta\Xi,\eta+s\epsilon\delta H)
		\left(h\phz\partial_{\alpha}\phi_i\right)\right]\:ds.
		\end{align*}
	
	To pass to the limit $$\lim\limits_{\epsilon\to 0}\int J^{\epsilon}(v,\xi,\eta)\:dx,$$ we use the Dominated Convergence Theorem, for which we need to bound 
	$|J^{\epsilon}(v,\xi,\eta)|$ by an integrable function independent of $\epsilon$. Let  $\epsilon\in[0,\epsilon^*]$, for some value $\epsilon^*$ fixed ,
	and $\ds p'=\frac{p}{p-1}$ be the dual exponent of $p$, then
	\begin{align*}
	\left| \pge  \phz \right|^{p'}
	&= \left|\frac{\partial\hat{e}}{\partial F_{i\alpha}}(\xi,\eta)
	+\frac{\partial\hat{e}}{\partial\zeta_{k\gamma}}(\xi,\eta)\frac{\partial(\mathrm{cof}F^0)^{k\gamma}}{\partial F_{i\alpha}}
	+\frac{\partial\hat{e}}{\partial w}(\xi,\eta)\frac{\partial(\mathrm{det}F^0)}{\partial F_{i\alpha}}\right|^{\frac{p}{p-1}}
	\\[4pt]
	&\leq C_1 \left(|\partial_F\hat{e}|^{\frac{p}{p-1}}+|\partial_{\zeta}\hat{e}|^{\frac{p}{p-1}}|F^0|^{\frac{p}{p-1}}+|\partial_w\hat{e}|^{\frac{p}{p-1}}|F^0|^{\frac{2p}{p-1}}\right).
	\end{align*}
	Using Young's inequality we can bound the terms
	\begin{align*}
	|\partial_{\zeta}\hat{e}|^{\frac{p}{p-1}}|F^0|^{\frac{p}{p-1}}\leq\frac{(|\partial_{\zeta}\hat{e}|^{\frac{p}{p-1}})^l}{l}+\frac{(|F^0|^{\frac{p}{p-1}})^{l'}}{l'}
	=\frac{|\partial_{\zeta}\hat{e}|^{\frac{p}{p-2}}}{l}+\frac{|F^0|^p}{l'}
	\end{align*}
	and
	\begin{align*}
	|\partial_w\hat{e}|^{\frac{p}{p-1}}|F^0|^{\frac{2p}{p-1}}\leq\frac{(|\partial_w\hat{e}|^{\frac{p}{p-1}})^l}{l}+\frac{(|F^0|^{\frac{2p}{p-1}})^{l'}}{l'}
	=\frac{|\partial_w\hat{e}|^{\frac{p}{p-3}}}{l}+\frac{|F^0|^p}{l'}
	\end{align*}
	for $l=\frac{p-1}{p-2}$ and $l'=p-1,$ so that the growth condition \eqref{growth.con.3} implies
	\begin{align}
	\label{bound_p_prime}
	\begin{split}
	\left|\pge\phz\right|^{p'}
	&\leq C_2 \left(|\partial_F\hat{e}|^{\frac{p}{p-1}}+|\partial_{\zeta}\hat{e}|^{\frac{p}{p-2}}+|\partial_w\hat{e}|^{\frac{p}{p-3}}+|F^0|^p\right)\\
	&\leq C_3 \left(|F|^p+|\zeta|^q+|w|^{\rho}+|\eta|^{\ell}+|F^0|^p +1\right).
	\end{split}
	\end{align}
	Therefore indeed $$\lim\limits_{\epsilon\to 0}\int J^{\epsilon}(v,\xi,\eta)\:dx=\int J(v,\xi,\eta)\:dx$$ and as a result
	$$\int(v_i-v_i^0)\phi_i+h\pge\phz\partial_{\alpha}\phi_i\:dx=0,$$ which yields the Euler-Lagrange equation \eqref{E-L_min}.
	
	\medskip
	\textbf{Step 4:} (On condition \eqref{constraint}). Consider the solution operator $S_h:(v^0,F^0,\zeta^0,w^0,\eta^0)\mapsto(v,F,\zeta,w,\eta)$ defined by equations 
	\eqref{discrete_thermoelasticity}. We want to validate \eqref{constraint} given that the minimizer $F=\nabla y$ and vice versa. Indeed, \eqref{lemma_weak_formulation}$_1$
	implies that $$\int F_{i\alpha}\phi\:dx=\int F^0_{i\alpha}\phi\:dx-\int hv_i\partial_{\alpha}\phi\:dx,$$ hence
	\begin{align*}
	\int F_{i\alpha}\partial_{\beta}\phi\:dx
	&=\int \partial_{\alpha}y^0_i\partial_{\beta}\phi\:dx-\int hv_i\partial_{\alpha}\partial_{\beta}\phi\:dx\\
	&=\int \partial_{\beta}y^0_i\partial_{\alpha}\phi\:dx-\int hv_i\partial_{\beta}\partial_{\alpha}\phi\:dx
	=\int F_{i\beta}\partial_{\alpha}\phi\:dx,
	\end{align*}
	for all $\phi\in C^{\infty}(\mathbb{T}^3).$
	Conversely, condition \eqref{constraint} implies that $F$ is conservative \textit{i.e.} $\hbox{curl}\,F=0$ therefore there exists $y\in\mathbb{R}^3$ such that $F=\nabla y$
	and since $F\in L^p(\mathbb{T}^3),$ $y\in W^{1,p}(\mathbb{T}^3).$
\end{proof}

\begin{lemma}
	\label{Lemma_null_lagrangian_WEAK}
	Let $p \ge 4$,
	\begin{align}
	\label{regularity_lemma_null_lagrangian_WEAK}
	y\in W^{1,\infty}(L^2(\mathbb{T}^3)) \cap L^\infty (W^{1,p} (\mathbb{T}^3))\quad\text{so that}\quad
	F_{i\alpha}\in L^{\infty}(L^p(\mathbb{T}^3))
	\end{align}
	and assume the field $g(x,t):\mathbb{R}^3\times\mathbb{R}^+\to\mathbb{R}^3$ such that
	\begin{align*}
	%\label{regularity_g_null_lagrangian_WEAK}
	\quad g_i\in L^{\infty}(L^2(\mathbb{T}^3)),\quad \partial_{\alpha}g_i\in L^{\infty}(L^p(\mathbb{T}^3))\quad (i=1,2,3).
	\end{align*}
	Then the identities
	\begin{equation}
	\label{identities_lemma_null_lagrangian_WEAK}
	\partial_{\alpha}\left(\ph g_i\right)=\ph \partial_{\alpha}g_i
	\end{equation}
	hold in the sense of distributions.
\end{lemma}
\begin{proof}
	Fix time, let $y$ as in \eqref{regularity_lemma_null_lagrangian_WEAK}, and consider the smooth convolution (in space)
	$y^{\epsilon}:=y\star \rho_{\epsilon}$,
	\tcb{ 
		where $\rho_{\epsilon}=\frac{1}{\epsilon^3} \rho(\frac{x}{\epsilon}),$  $0<\rho\in C^{\infty}_c (\mathbb{R}^3),$
		$\int\rho(x)dx =1$,
	}
	then  $y^{\epsilon}\in C^{\infty}(\mathbb{T}^3)$,
	$F_{i\alpha}^{\epsilon}=\partial_{\alpha}y^{\epsilon}_i\in C^{\infty}(\mathbb{T}^3)$ and 
	\begin{align*}
	\|y^{\epsilon}-y\|_{W^{1,p}}\to 0 \, .
	\end{align*}
	Similarly, define the convolution of $g$ with
	a smooth kernel: $g^{\epsilon}:=g\star\psi_{\epsilon}$ so that 
	$g^{\epsilon},\partial_{\alpha}g^{\epsilon}\in C^{\infty}(\mathbb{T}^3).$ Then
	\begin{align}
	\label{convergence_g}
	\|g^{\epsilon}-g\|_{L^2}\to 0,\quad \|\partial_{\alpha}g^{\epsilon}-\partial_{\alpha}g\|_{L^p}\to 0.
	\end{align}
	As the cofactor function \eqref{cofactor_d=3} is bilinear
	in $F$ and the determinant \eqref{det_d=3} trilinear ($d=3$), we can obtain the bounds 
	\begin{align*}
	\left\|\frac{\partial\cof F^{\epsilon}}{\partial F_{i\alpha}}-\frac{\partial\cof F}{\partial F_{i\alpha}}\right\|_{L^p}
	\leq\left\|F^{\epsilon}-F\right\|_{L^p}
	\end{align*}
	and
	\begin{align*}
	\left\|\frac{\partial\det F^{\epsilon}}{\partial F_{i\alpha}}-\frac{\partial\det F}{\partial F_{i\alpha}}\right\|_{L^{p/2}}
	&=\left\|\frac{1}{3}(\cof F^{\epsilon})_{i\alpha}-\frac{1}{3}(\cof F)_{i\alpha}\right\|_{L^{p/2}}\\
	&=\left\|\frac{1}{3}\epsilon_{ijk}\epsilon_{\alpha\beta\gamma}(F^{\epsilon}_{j\beta}F^{\epsilon}_{k\gamma}-F_{j\beta}F_{k\gamma})\right\|_{L^{p/2}}\\
	&\leq c\left\|F^{\epsilon}-F\right\|_{L^p}\left\||F^{\epsilon}|+|F|\right\|_{L^p}
	\end{align*}
	which in turn imply
	\begin{align}
	\label{convergence_cof}
	\frac{\partial\cof F^{\epsilon}}{\partial F_{i\alpha}}&\to \frac{\partial\cof F}{\partial F_{i\alpha}} \quad\text{in}\:\: L^p(\mathbb{T}^3)\\
	\label{convergence_det}
	\frac{\partial\det F^{\epsilon}}{\partial F_{i\alpha}}&\to \frac{\partial\det F}{\partial F_{i\alpha}} \quad\text{in}\:\: L^{p/2}(\mathbb{T}^3).
	\end{align}
	By \eqref{Euler-Lagrange}, in order to show \eqref{identities_lemma_null_lagrangian_WEAK}, it suffices to pass to the limit in the identities 
	\begin{align}
	\label{epsilon_eqn_cof}
	\int\epsilon_{ijk}\epsilon_{\alpha\beta\gamma}F_{j\beta}^{\epsilon}g_i^{\epsilon}\partial_{\alpha}\phi_i\:dx&=
	-\int\epsilon_{ijk}\epsilon_{\alpha\beta\gamma}F_{j\beta}^{\epsilon}\partial_{\alpha}g_i^{\epsilon}\phi_i\:dx\\
	\label{epsilon_eqn_det}
	\int\epsilon_{ijk}\epsilon_{\alpha\beta\gamma}F_{j\beta}^{\epsilon}F_{k\gamma}^{\epsilon}g_i^{\epsilon}\partial_{\alpha}\phi_i\:dx&=
	-\int\epsilon_{ijk}\epsilon_{\alpha\beta\gamma}F_{j\beta}^{\epsilon}F_{k\gamma}^{\epsilon}\partial_{\alpha}g_i^{\epsilon}\phi_i\:dx
	\end{align}
	For \eqref{epsilon_eqn_cof}, observe that $p'=\frac{p}{p-1}\leq 2$ when $p>4.$ Therefore,
	$g_i^{\epsilon}\to g_i $ in $L^{p'}(\mathbb{T}^3)$ and by \eqref{convergence_cof} the product $F_{j\beta}g_i\in L^1$.
	For the right hand side, due to
	$F^{\epsilon}\to F$ in $L^p(\mathbb{T}^3)$ and $\partial_{\alpha}g^{\epsilon}\to \partial_{\alpha}g$ in $L^p(\mathbb{T}^3),$ their product converges
	in $L^{p/2}.$ For \eqref{epsilon_eqn_det}, we have that $F_{j\beta}^{\epsilon}F_{k\gamma}^{\epsilon}\to F_{j\beta}F_{k\gamma}$
	in $L^{p/2}$ this is exactly \eqref{convergence_det}. For $p\ge 4$, the dual exponent of $\frac{p}{2}$ namely
	$\left ( \tfrac{p}{2} \right )' = \frac{p}{p-2}\leq 2$ 
	so that $g^{\epsilon}\to g$ in $L^{\frac{p}{p-2}}(\mathbb{T}^3).$ As a result the product $F_{j\beta}F_{k\gamma}g_i\in L^1.$
	Similarly, using \eqref{convergence_g}$_2$ and because $p\geq \frac{p}{2}$ we deduce that $\partial_{\alpha}g^{\epsilon}\to \partial_{\alpha}g$
	in $L^{p/2}(\mathbb{T}^3).$ Then given \eqref{convergence_det}, in the limit $F_{j\beta}F_{k\gamma}\partial_{\alpha}g_i\in L^{p/4}.$
	This shows that we can pass to the limit as $\epsilon\to 0$ and completes the proof.
\end{proof}

The next step is to prove that the scheme dissipates the energy. For this we need a uniform convexity hypothesis for $\hat{e}(\xi,\eta)$,
namely there exists a constant $c > 0$ such that 
	\begin{equation}
	\label{Uniform_Convexity}
	\nabla_{(\xi,\eta)}^2\hat{e}\geq\ c \; \mathbb{I} \, .
	\end{equation}
\begin{lemma}\label{lemmaenergydiss}
	\label{dissipation}	
	Let $U^0=(v^0,\xi^0,\eta^0),$ $U=(v,\xi,\eta)$ $\in(L^2\times L_{p,q,\rho}\times L^{\ell})(\mathbb{T}^3)$ and set
	$I(U)=I(v,\xi,\eta)$ given by \eqref{I(U)}.
	If the mapping $(\xi,\eta)\mapsto\hat{e}(\xi,\eta)$ is uniformly convex 
	it holds
	\begin{align}
	\label{relative_energy_estimate}
	\int\left(I(U)+c|U-U^0|^2+h\hat{\theta}(\xi,\eta)\frac{r}{\hat{\theta}(\xi^0,\eta^0)}\right)\:dx\leq\int I(U^0)\:dx,
	\end{align}
	for some numerical constant $c.$ As a result, there exist constants $C=C(U^0,r)$ and $E=E(U^0)$ such that
	\begin{align}
	\begin{split}
	\label{uniform_energy_estimate6}
	\sup\limits_j&\left(\|v^j\|^2_{L^2}+\int\hat{e}(\xi^j,\eta^j)\:dx+C\int h\hat{\theta}(\xi^j,\eta^j)\:dx\right)\\
	&+\sum_{j=1}^{\infty}\left(\|v^j-v^{j-1}\|^2_{L^2}+\|\xi^j-\xi^{j-1}\|^2_{L^2}+\|\eta^j-\eta^{j-1}\|^2_{L^2}\right)\leq\:E.
	\end{split}
	\end{align}
\end{lemma}
\begin{proof}
	We calculate
	\begin{align}
	\label{lemma_2_calculation}
	\begin{split}
	I(U)-&I(U^0)-DI(U)\cdot(U-U^0)\\
	&=\frac{|v|^2}{2}-\frac{|v^0|^2}{2}+\hat{e}(\xi,\eta)-\hat{e}(\xi^0,\eta^0)-(v,\hat{e}_{\xi}(\xi,\eta),\hat{e}_{\eta}(\xi,\eta))\cdot(v-v^0,\xi-\xi^0,\eta-\eta^0)
	\end{split}
	\end{align}
	then using \eqref{discrete_thermoelasticity}, \eqref{e:constitutive_functions} and the null-Lagrangian property \eqref{Euler-Lagrange} we further compute
	\begin{align}
	\label{lemma2_calculation}
	\begin{split}
	\frac{1}{h}DI(U)&\cdot(U-U^0)=\frac{1}{h}\left(v_i(v_i-v_i^0)+\pge(\xi-\xi^0)^B+\frac{\partial\hat{e}}{\partial\eta}(\xi,\eta)(\eta-\eta^0)\right)\\
	&=v_i\partial_{\alpha}\left(\pge\phz\right)+\pge\partial_{\alpha}\left(\phz v_i\right)+\hat{\theta}(\xi,\eta)\frac{r}{\hat{\theta}(\xi^0,\eta^0)}\\
	&=v_i\partial_{\alpha}\left(\pge\phz\right)+\pge\phz \partial_{\alpha}v_i+\hat{\theta}(\xi,\eta)\frac{r}{\hat{\theta}(\xi^0,\eta^0)}\\
	&=\partial_{\alpha}\left(\pge\phz v_i\right)+\hat{\theta}(\xi,\eta)\frac{r}{\hat{\theta}(\xi^0,\eta^0)}.
	\end{split}
	\end{align}
	Observe that the third equality holds as a result of Lemma \ref{Lemma_null_lagrangian_WEAK}.
	To validate the last equality in the regularity class $(L^2\times L_{p,q,\rho}\times L^{\ell})(\mathbb{T}^3),$ we infer that the functions 
	$\ds\pge\phz$ and $v_i$ are weakly differentiable with 
	\begin{align}
	\label{lem32p}
	\partial_{\alpha}\left(\pge\phz\right)\in L^2(\mathbb{T}^3) \quad\text{and}\quad \partial_{\alpha}v_i\in L^p(\mathbb{T}^3)\,.
	\end{align}
	The first statement is due to the fact $v_i-v_i^0\in L^2(\mathbb{T}^3)$ and equation \eqref{discrete_thermoelasticity}$_2$ while the second comes from
	equation \eqref{discrete_thermoelasticity}$_1$ and in particular $$\frac{F_{i\alpha}-F_{i\alpha}^0}{h}=\partial_{\alpha}v_i$$ given that
	$F_{i\alpha}-F_{i\alpha}^0\in L^p(\mathbb{T}^3).$ Bound \eqref{bound_p_prime}
	implies 
	\begin{align*}
	\left|\pge\phz\right|^{p'}
	&\leq C \left(|F^0|^p+|F|^p+|\zeta|^q+|w|^{\rho}+|\eta|^{\ell}+1\right),
	\end{align*}
	therefore 
	\begin{align*}
	\pge\phz\in L^{p'}(\mathbb{T}^3).
	\end{align*}
	This along with \eqref{lem32p}$_2$ asserts that the product $$\ds \pge\phz \partial_{\alpha}v_i\in L^1(\mathbb{T}^3)$$ and similarly, because of \eqref{lem32p}$_1$,
	$$\ds\partial_{\alpha}v_i\left(\pge\phz\right)\in L^1(\mathbb{T}^3).$$ Also, due to \eqref{lem32p}$_2$ by Poincar\'{e}'s inequality we get that $v_i\in W^{1,p}(\mathbb{T}^3)$
	so that $v_i\in L^2(\mathbb{T}^3)\cap W^{1,p}(\mathbb{T}^3).$ This ensures the term $$\pge\phz v_i\in L^1(\mathbb{T}^3)$$
	and validates computation \eqref{lemma2_calculation}.

	Now observe that $I$ is uniformly convex and as a result 
	\begin{align*}
	\int I(U)&-I(U^0)-DI(U)\cdot(U-U^0)\:dx\\
	&=-\int\int_0^1\int_0^1 (1-s) D^2 I(U-\tau(1-s)(U-U^0))\cdot(U-U^0,U-U^0)\:ds\:d\tau\:dx\\
	&\leq-c\int|U-U^0|^2\:dx\,,
	\end{align*}
	where $\ds c=\min_{\mathcal{C}}\{D^2I(U)\}$. Therefore \eqref{lemma_2_calculation} and \eqref{lemma2_calculation} yield the bound
	\begin{small}
		\begin{align*}
		\int\left(\frac{|v|^2}{2}-\frac{|v^0|^2}{2}+\hat{e}(\xi,\eta)-\hat{e}(\xi^0,\eta^0)
		-h\partial_{\alpha}\left(\pge\phz v_i\right)-h\hat{\theta}(\xi,\eta)\frac{r}{\hat{\theta}(\xi^0,\eta^0)}\right)\:dx\\
		\leq -c\int|U-U^0|^2\:dx
		\end{align*}
	\end{small}
	\newline
	and since the divergence term integrates to zero, we obtain \eqref{relative_energy_estimate}.
	Finally, uniform bound \eqref{uniform_energy_estimate6} follows by simply rearranging the terms in \eqref{relative_energy_estimate} written for $(v^j,\xi^j,\eta^j)$
	and $(v^{j-1},\xi^{j-1},\eta^{j-1}),$ that is
	\begin{align*}
	\int\left(\frac{|v^j|^2}{2}+\hat{e}(\xi^j,\eta^j)+c|U^j-U^{j-1}|^2+h\hat{\theta}(\xi^j,\eta^j)\frac{r}{\hat{\theta}(\xi^{j-1},\eta^{j-1})}\right)\:dx\\
	\leq\int\left(\frac{|v^{j-1}|^2}{2}+\hat{e}(\xi^{j-1},\eta^{j-1})\right)\:dx
	\end{align*}
	and summing in $j.$
\end{proof}

As a corollary to Theorem \ref{minimization} and Lemma \ref{dissipation} we can show that the minimizer also satisfies the energy inequality
\eqref{dicrete_energy_ineq} in the sense of distributions.
\begin{corollary}
	\label{minimizer_energy}
	The minimizer constructed in Theorem \ref{minimization} satisfies the energy inequality
	\begin{align}
	\label{weak_enery_inequality}
	\begin{split}
	\int\phi\frac{1}{h}\left(\frac{1}{2}|v|^2+\right.&\left.\hat{e}(\xi,\eta)-\frac{1}{2}|v^0|^2-\hat{e}(\xi^0,\eta^0)\right)\:dx\leq\\
	&\leq-\int\left(\pge\phz v_i\right)\partial_{\alpha}\phi\:dx
	+\int\phi\:\hat{\theta}(\xi,\eta)\frac{r}{\hat{\theta}(\xi^0,\eta^0)}\:dx,
	\end{split}
	\end{align}
	where $0 \le \phi\in C^{\infty}(\mathbb{T}^3).$
\end{corollary}
\begin{proof}
	Notice that, using the notation of Lemma \ref{dissipation}
	\begin{equation}
	\label{corollary_EQN1}
	DI(U)\cdot(U-U^0)=I(U)-I(U^0)+I(U^0|U),
	\end{equation}
	this is \eqref{relative_I(U)}.	The last term in \eqref{corollary_EQN1}	is the quadratic part of the Taylor expansion. Inferring again
	that $I(U^0|U)\geq 0$ because $I(U)$ is convex, we get the inequality
	$$I(U)-I(U^0)\leq DI(U)\cdot(U-U^0)$$
	holding in the sense of distributions. Substituting from \eqref{I(U)} and \eqref{lemma2_calculation}, we immediately obtain
	\eqref{weak_enery_inequality}.
\end{proof}

%
%Section 5
%

\section{A dissipative measure-valued solution for polyconvex thermoelasticity}\label{3:SEC5}
Let $(v^j,\xi^j,\eta^j)$ on $\mathbb{T}^3$ be the iterates constructed by solving the minimization problem, given $(v^0,\xi^0,\eta^0),$ $j=0,1,2,\dots$ according to Theorem 
\ref{minimization}. Then $F^j$ are gradients if $F^0$ is a gradient and at each time step we construct a function $y^j:\mathbb{T}^3\to\mathbb{R}^3$ such that 
$\partial_{\alpha}y_i^j=F_{i\alpha}^j.$ Choosing $y^{-1}$ by extrapolation, the iterates satisfy $$\frac{1}{h}(y^j-y^{j-1})=v^j.$$ 

In what follows we fix the notation $Q_{T}=\mathbb{T}^3\times[0,T]$
and state our main result. The rest of this Section will serve to prove this Theorem.
\begin{theorem}
	\label{mainconv}
	Assume \eqref{Uniform_Convexity}, that for some $\delta > 0$
	\begin{equation}
	\label{lowerb}
	\hat \theta (\xi, \eta) = \frac{\del \hat e}{\del \eta}(\xi, \eta) \ge \delta >0 \,  ,
	\end{equation}
	and that the growth conditions  \eqref{growth.con.1}, \eqref{growth.con.2} and \eqref{growth.con.3}  are satisfied
	with exponents
	\begin{align}
	\label{exponents}
	\qquad\qquad p\geq 4 \, , \quad q\geq 2 \, , \quad \rho> 1 \, , \quad  \ell>1 \, . 
	\end{align}
	In the limit $h\to 0$, 	the variational scheme \eqref{discrete_thermoelasticity_j}, solved by the minimization method of Theorem \ref{minimization},
	%described in Sections \ref{3:SEC3} and \ref{3:SEC4}  
	generates a dissipative measure-valued solution for the adiabatic polyconvex thermoelasticity system
	(\ref{e:adiabatic_thermoelasticity_Phi}) consisting of a thermomechanical process
	$(y(t,x),\eta(t,x)):[0, T]\times\mathbb{T}^3\to\mathbb{R}^3\times\mathbb{R}$ for any $T > 0$,
	\begin{align}
	\label{e:regulatity.y} 
	y \in W^{1,\infty}(L^2(\mathbb{T}^3)) \cap L^\infty (W^{1,p} (\mathbb{T}^3)) \, , \quad \eta\in L^{\infty}(L^{\ell} (\mathbb{T}^3)   ) \,  ,
	\end{align}
	a parametrized family of probability measures $\bn=\bn_{(x,t)\in Q_{T}}$ and a nonnegative Radon measure $\bg\in\mathcal{M}^+(Q_{T}).$
	If $(v,F,\eta)$ denote the averages
		\begin{align*}
		F=\left\la\bn,\lambda_F \right\ra, %\quad \zeta=\left\la\bn,\lambda_{\zeta} \right\ra,\quad w=\left\la\bn,\lambda_w \right\ra,
		\quad v=\left\la\bn,\lv\right\ra, \quad \eta=\left\la\bn,\li\right\ra \, , 
		\end{align*} 
	then $\bn$ and $\bg$ satisfy
	\begin{align}
	\label{e:regulatity.xi-v-eta}
	F=\nabla y\in L^{\infty}(L^p), \quad  v=\partial_t y\in L^{\infty}(L^2) \, , 
	\end{align}
	$$\Phi(F)=(F, \mathrm{cof} F, \det F)\in L^{\infty}(L^p)\times L^{\infty}(L^q)\times L^{\infty}(L^{\rho}) \, , $$
	and the averaged equations
	\begin{align}
	\begin{split}
	\label{e:mv.sol.xi-v-eta}
	\partial_t \Phi^B(F)&=\partial_{\alpha}\left(\phv \right) \\ 
	\partial_t \left\la\bn,\li\right\ra &= \left\la\bn,\frac{r}{\hat{\theta}(\Phi(\lambda_F),\li)}\right\ra\\
	\partial_t \left\la\bn,\lvi\right\ra
	&=\partial_{\alpha}\left\la\bn,\frac{\partial \hat{e}}{\partial\xi^B}(\Phi(\lambda_F),\lambda_{\eta})\frac{\partial\Phi^B}{\partial F_{i\alpha}}(\lf)\right\ra
	\end{split}
	\end{align}
	in the sense of  distributions. Moreover, they satisfy the integrated form of the  averaged energy inequality,
	\begin{small}
		\begin{align}
		\begin{split}
		\label{e:mv.sol.Fvt.energy}
		-\int& \varphi(0)\left\la\bn,\frac{1}{2}|\lv|^2+\hat{e}(\Phi(\lambda_F),\li)\right\ra (x,0)\:dx \\ %+\bg_0(dx) \right)\\
		&-\int_0^T
		\int  \varphi^{\prime}(t) \left(\left\la\bn,\frac{1}{2}|\lv|^2+\hat{e}(\Phi(\lambda_F),\li) \right\ra(x,t)  \, dx\:dt +\bg(dx\:dt) \right)  \\
		&\leq\int_0^T\int \left\la\bn,r\right\ra\varphi(t)\:dx\:dt,
		\end{split}
		\end{align}
	\end{small}
	holding for all $\varphi \in C^{\infty}_c([0,T))$, $\varphi\ge 0$.
\end{theorem} 

\smallskip
\tcb{ The notation $\big\langle \bn , \psi (\lambda_v ,  \lambda_\xi , \lambda_\eta ) \big \rangle$ is used to denote the action of the Young measure
	$\bn$ on the function $\psi (v, \xi, \eta)$, with $\xi = (F, \zeta, w)$ the vector of extended variables. The definition of $\nu$ is recalled in the text following 
	Lemma \ref{lemma.weak.minors_TIME_DERIVATIVES}.}

We define $(V^h,\Xi^h,H^h)$ to be the approximate solution constructed via piecewise linear interpolation of the iterates
$(v^j(x),\xi^j(x),\eta^j(x)),$ $j=1,\dots,N:$
\begin{align}
\label{pws_linear}
\begin{split}
V^h(t)=\sum_{j=1}^{\infty}\chi^j(t)\left(v^{j-1}+\frac{t-h(j-1)}{h}(v^j-v^{j-1})\right)\\
\Xi^h(t)=(F^h,Z^h,W^h)(t)=\sum_{j=1}^{\infty}\chi^j(t)\left(\xi^{j-1}+\frac{t-h(j-1)}{h}(\xi^j-\xi^{j-1})\right)\\
H^h(t)=\sum_{j=1}^{\infty}\chi^j(t)\left(\eta^{j-1}+\frac{t-h(j-1)}{h}(\eta^j-\eta^{j-1})\right).
\end{split}
\end{align}
We also define $(v^h,\xi^h,\eta^h)$ to be the piecewise constant interpolation
\begin{align}
\label{pws_constant}
\begin{split}
v^h(t)=\sum_{j=1}^{\infty}\chi^j(t)v^j\\
\xi^h(t)=(f^h,\zeta^h,w^h)(t)=\sum_{j=1}^{\infty}\chi^j(t)\xi^j\\
\eta^h(t)=\sum_{j=1}^{\infty}\chi^j(t)\eta^j
\end{split}
\end{align}
where in both cases $\chi^j(t):=\chi_{[(j-1)h,jh]}$ stands for the characteristic function over the time interval $[(j-1)h,jh]$. For convenience, let us also denote 
\begin{align*}
v^h(t-h)=\sum_{j=1}^{\infty}\chi^j(t)v^{j-1}\\
\xi^h(t-h)=(f^h,\zeta^h,w^h)(t-h)=\sum_{j=1}^{\infty}\chi^j(t)\xi^{j-1}\\
\eta^h(t-h)=\sum_{j=1}^{\infty}\chi^j(t)\eta^{j-1}.
\end{align*}
Subsequently, one can also write a piecewise linear approximation of the motion as follows
\begin{align}
\label{pws_linear_Y}
Y^h(t)=\sum_{j=1}^{\infty}\chi^j(t)\left(y^{j-1}+\frac{t-h(j-1)}{h}(y^j-y^{j-1})\right).
\end{align}
Note that a direct computation gives
\begin{align}
\label{derivatives_Y^h}
\partial_t Y_i^h=v_i^h \quad\text{and}\quad \partial_{\alpha} Y_i^h=F_{i\alpha}^h.
\end{align}
Viewing $\{(V^h,\Xi^h,H^h)\}$ and $\{(v^h,\xi^h,\eta^h)\}$ as sequences in $h>0$, our goal is to prove that in the limit as the 
time step $h\to 0,$ we can obtain a dissipative measure valued solution to adiabatic thermoelasticity.
Lemma \ref{dissipation} and in particular bound \eqref{uniform_energy_estimate6} along with 
\eqref{growth.con.1} imply the uniform bound on the iterates
$(v^j,\xi^j,\eta^j):$
\begin{align*}
%\label{unif_6+GC1}
\sup\limits_j\int|v^j|^2+\left(|F^j|^p+|\zeta^j|^q+|w^j|^{\rho}+|\eta^j|^{\ell}-1\right)dx
\leq \sup\limits_j\left(\|v^j\|^2_{L^2}+\int\hat{e}(\xi^j,\eta^j)dx\right)\leq\:E.
\end{align*}
By this estimate and by convexity of the $L^p$ norms involved, one arrives to the conclusion that the sequences
\begin{align}
\label{h-sequenceses-boundness}
\{(V^h,\Xi^h,H^h)\}\;\;\text{and}\;\;\{(v^h,\xi^h,\eta^h)\}\;\;\text{are uniformly bounded in}\;\;
L^{\infty}(L^2\times L_{p,q,\rho}\times L^{\ell}(\mathbb{T}^3)).
\end{align} 
Therefore, up to a subsequence, they converge weakly-$\ast$ 
in $L_{loc}^{\infty}((L^2\times L_{p,q,\rho}\times L^{\ell})(\mathbb{T}^3)).$
Additionally \eqref{h-sequenceses-boundness} also implies that the sequence $\{Y^h\}$ is 
bounded in $W^{1,\infty}(L^2(\mathbb{T}^3)) \cap L^\infty (W^{1,p} (\mathbb{T}^3))$
which in turn suggests that $Y^h$ converges weakly in $H_{loc}^1([0,\infty)\times\mathbb{T}^3)$
along subsequences and by Rellich's Theorem converges strongly in $L^2.$
%\begin{align}
%Y^h\to y&\quad\text{strongly in\:}L^2_{loc}(\mathbb{T}^3),\:\text{a.e in\:}x\\
%f_{i\alpha}^h\rightharpoonup F_{i\alpha},\:\:F_{i\alpha}^h\rightharpoonup F_{i\alpha}
%&\quad\text{weak-$\ast$ in\:}L^{\infty}_{loc}(L^p(\mathbb{T}^3)),\\
%v_i^h\rightharpoonup v,\:\:V_i^h\rightharpoonup v
%&\quad\text{weak-$\ast$ in\:}L^{\infty}_{loc}(L^2(\mathbb{T}^3))\\
%\eta^h\rightharpoonup \eta,\:\:H^h\rightharpoonup \eta
%&\quad\text{weak-$\ast$ in\:}L^{\infty}_{loc}(L^{\ell}(\mathbb{T}^3)).
%\end{align}

In \eqref{e:mv.sol.xi-v-eta}, the first equation holds in the classical weak sense. This 
is an immediate implication of the weak continuity of null-Lagrangians \cite[Lemma 6.1]{MR0475169}
and \cite[Lemmas 4,5]{MR1831179} according to which in the regularity class \eqref{e:regulatity.y} and for exponents
as in \eqref{exponents}, $\mathrm{cof}\nabla y$ and $\mathrm{det}\nabla y$ are weakly continuous: 
Written explicitly as in \eqref{cofactor_d=3},~\eqref{det_d=3} (for $d=3$), we have that if a sequence 
\begin{align*}
%\label{y.reg}
y_n \; \mbox{is bounded in } \; W^{1,\infty}(L^2(\mathbb{T}^3)) \cap L^\infty (W^{1,p} (\mathbb{T}^3)) \, ,
\end{align*}
then along a subsequence
\begin{align*}
(F_n,\mathrm{cof}F_n,\det F_n)\rightharpoonup(F,\mathrm{cof}F,\det F), 
\:\:\text{weak-$\ast$ in }\:\: L^{\infty}(L^p)\times L^{\infty}(L^q)\times L^{\infty}(L^{\rho})
\end{align*}
for $p\ge 4,\:q\ge\frac{p}{p-1},\:q\ge\frac{4}{3}$, $\rho>1$ \cite[Lemma 3]{MR1831179}.
This result is valid under the regularity conditions \eqref{e:regulatity.y} and for exponents: $p\geq 4,$ $q\geq2$, $\rho > 1$,
and it allows to take the limit
\begin{align}
\label{h_limit_F-cofF-detF}
\begin{split}
(F^h,\mathrm{cof}F^h,\det F^h)&\rightharpoonup(F, \mathrm{cof}F,\det F)  ,\:\text{weak-$\ast$ in\:} 
L^{\infty}(L_{p,q,\rho})(\mathbb{T}^3)\\
(f^h,\mathrm{cof}f^h,\det f^h)&\rightharpoonup(F, \mathrm{cof}F,\det F)  ,\:\text{weak-$\ast$ in\:} 
L^{\infty}(L_{p,q,\rho})(\mathbb{T}^3).
\end{split}
\end{align}
In addition, the geometric constraints \eqref{transportQIN} are stable under weak convergence 
\begin{align*}
Y^h\to y&\quad\text{strongly in\:}L^2_{loc}(\mathbb{T}^3),\:\text{a.e in\:}x\\
f_{i\alpha}^h\rightharpoonup F_{i\alpha},\:\:F_{i\alpha}^h\rightharpoonup F_{i\alpha}
&\quad\text{weak-$\ast$ in\:}L^{\infty}_{loc}(L^p(\mathbb{T}^3)),\\
v_i^h\rightharpoonup v
&\quad\text{weak-$\ast$ in\:}L^{\infty}_{loc}(L^2(\mathbb{T}^3))
\end{align*}
in the so-called regularity framework \cite[Lemma 4]{MR1831179}. In turn, this allows us to pass to the limit in the 
transport identities
\begin{align*}
\partial_t F^h=&\partial_\alpha v_i^h\\
\partial_t(\mathrm{cof}F^h)_{k \gamma}
&=\partial_t\partial_{\alpha}\left(\frac{1}{2}\epsilon_{ijk}\epsilon_{\alpha\beta\gamma}Y^h_{i}F^h_{j\beta}\right)
= \partial_{\alpha}(\epsilon_{ijk}\epsilon_{\alpha\beta\gamma}F^h_{j\beta}v^h_{i}),\\
\partial_t\det F^h
&=\partial_t\partial_\alpha\left(\frac{1}{3}Y^h_{i}(\mathrm{cof}F^h)_{i \alpha}\right)
=\partial_\alpha \big(\mathrm{cof}F^h_{i \alpha}v_i^h \big)
\end{align*}
and obtain 
\begin{equation}
\label{transiden}
\begin{aligned}
\partial_t F_{i\alpha}&=\partial_{\alpha}v_i \\ 
\partial_t\det F&=\partial_{\alpha}\bigl((\mathrm{cof} F)_{i\alpha}v_i\bigr)\\
\partial_t(\mathrm{cof} F)_{k\gamma}&=\partial_{\alpha}(\epsilon_{ijk}\epsilon_{\alpha\beta\gamma}F_{j\beta}v_i).
\end{aligned}
\end{equation}

Moreover, if $F^h,\:Z^h,\:W^h$ are as in \eqref{pws_linear} and produced by the minimization scheme then 
we have the following convergence result on their time derivatives:  
\begin{lemma} \label{lemma.weak.minors_TIME_DERIVATIVES}\cite[Lemma 5]{MR1831179}
	There holds:
	\begin{equation}
	\label{dt_Z,W_convergence}
	\begin{aligned}
	\partial_t\left(Z^h-\mathrm{cof} F^h\right)\rightharpoonup 0\\
	\partial_t\left(W^h-\det F^h\right)\rightharpoonup 0
	\end{aligned}
	\qquad \text{in the sense of distibutions on } Q_{T}.
	\end{equation}
\end{lemma}
The sequences $\{(V^h,\Xi^h,H^h)\},$ $\{(v^h,\xi^h,\eta^h)\}$ generate a parametrized family of probability measures
$$\bn_{(x,t)\in Q_{T}}\in\mathcal{P}(\mathbb{R}^3\times\mathbb{M}^{3\times 3}\times\mathbb{M}^{3\times 3}\times\mathbb{R}\times\mathbb{R})$$
given by the mapping $\bn:Q_{T}\ni(x,t)\mapsto\bn_{(x,t)}.$ Both sequences generate the same $\bn$ (see below) which
is a weakly$-\ast$ measurable, essentially bounded
map - we will be denoting this in short as $L^{\infty}_{weak}$- representing weak limits of the form
\begin{align}
\label{psi_weaklim_definition}
\text{wk-$\ast$-}\lim_{h\to 0}\psi(v_h,\xi_h,\eta_h)
=\text{wk-$\ast$-}\lim_{h\to 0}\psi(v_h,F_h,\zeta_h,w_h,\eta_h)=\la\bn,\psi(\lv, \lambda_\xi,\li)\ra,
\end{align}
for all continuous functions $\psi=\psi(\lv,\lf,\lz,\lw,\li)$ with polynomial growth
\begin{align}
\label{psi_weaklim_growth}
\lim_{|\lv|^2+|\lf|^p+|\lz|^q+|\lw|^{\rho}+|\li|^{\ell}\to\infty}
\frac{|\psi(\lv,\lf,\lz,\lw,\li)|}{|\lv|^2+|\lf|^p+|\lz|^q+|\lw|^{\rho}+|\li|^{\ell}}=0,
\end{align}
where in \eqref{psi_weaklim_definition} the notation $\la\bn,\cdot\ra$ stands for the average
$$\la\bn,\psi(\lv,\lf,\lz,\lw,\li)\ra=\int\psi(\lv,\lf,\lz,\lw,\li)\:\bn(d\lv,d\lf,d\lz,d\lw,d\li)$$
and $\lv\in\mathbb{R}^3,$ $\lf\in\mathbb{M}^{3\times 3},$ $\lz\in\mathbb{M}^{3\times 3},$ $\lw\in\mathbb{R},$ 
$\li\in\mathbb{R}.$

We now verify that the sequences $(V^h,\Xi^h,H^h)$ and $(v^h,\xi^h,\eta^h)$ generate the same Young 
measure. 
Observe that according to \eqref{uniform_energy_estimate6}
$$\sum_{j=1}^{\infty}\left(\|v^j-v^{j-1}\|^2_{L^2}+\|\xi^j-\xi^{j-1}\|^2_{L^2}+\|\eta^j-\eta^{j-1}\|^2_{L^2}\right)
\leq\:E,$$ so one can obtain the bound:
\begin{align*}
\|F^h-f^h\|_{L^2(Q_{T})}&\leq\left(\sum_{j=1}^{\infty}\int_{0}^{T}\chi^j(t)\left|\frac{t-h(j-1)}{h}\right|^2
\int|F^j-F^{j-1}|^2\:dxdt\right)^{1/2}\\
&\leq\left(\frac{h}{3}\:\sum_{j=1}^{\infty}\int|F^j-F^{j-1}|^2\:dx\right)^{1/2}\leq\sqrt{hE}.
\end{align*}
	and since
	\begin{equation*}
	\|F^h\|_{L^{p}(Q_{T})}+\|f^h\|_{L^{p}(Q_{T})}<C
	\end{equation*}
	we have
	\begin{align}
	\label{convergence_F_L^pbar}
	\|F^h-f^h\|_{L^{\bar{p}}(Q_{T})} 
	\to 0,\qquad \bar{p}\in[2,p) \, .
	\end{align}
	Similarly,  when $q \ge 2$, $\rho \ge 2$, $\ell \ge 2$, we deduce
	\begin{equation}
	\label{convergence_Z,V_L^pbar}
	\|Z^h-\zeta^h\|_{L^{\bar{q}}(Q_{T})}\to 0\quad \bar{q}\in[2,q),\qquad
	\|V^h-v^h\|_{L^2(Q_{T})}\to 0,
	\end{equation}
	\begin{equation*}
	\|W^h-w^h\|_{L^{\bar{\rho}}(Q_{T})}\to 0\quad \bar{\rho}\in[2,\rho),\qquad 
	\|H^h-\eta^h\|_{L^{\bar{\ell}}(Q_{T})}\to 0\quad \bar{\ell}\in[2,\ell).
	\end{equation*}
	Additionally, because $Q_{T}$ is a bounded domain, it also implies
	\begin{align*}
	\|W^h-w^h\|_{L^{\bar\rho}(Q_{T})}&\leq\|W^h-w^h\|_{L^2(Q_{T})},\qquad\text{if}\:\:\bar\rho\leq 2\\
	\|H^h-\eta^h\|_{L^{\bar\ell}(Q_{T})}&\leq\|H^h-\eta^h\|_{L^2(Q_{T})},\qquad\text{if}\:\:\bar\ell\leq 2 \, ,
	\end{align*}
	which implies that, even in the case $1 < \rho < 2$ or $1 < \ell < 2$, it still holds
\begin{equation}
\label{convergence_W,H_L^pbar}
\|W^h-w^h\|_{L^{\rho}(Q_{T})}\to 0 \, ,  \quad 
\|H^h-\eta^h\|_{L^{\ell}(Q_{T})}\to 0 \, .
\end{equation}
	We conclude
	$$
	(V^h - v_h ,\Xi^h - \xi_h ,H^h - \eta_h ) \to 0 \quad \mbox{in measure}
	$$ 
	and as a result of \cite[Lemma 5.3]{MR1741439} the sequences $(V^h  ,\Xi^h  ,H^h  )$ and $(v_h , \xi_h , \eta_h )$
generate the same Young measure. 

Therefore, in the class of representable functions,
\begin{align}
\label{psi_weak-lim_INXI}
\begin{split}
\text{wk-$\ast$-}\lim_{h\to 0}\psi(V^h,\Xi^h,H^h)&=\text{wk-$\ast$-}\lim_{h\to 0}\psi(v^h,\xi^h,\eta^h)\\
&=\text{wk-$\ast$-}\lim_{h\to 0}\psi(v^h(t-h),\xi^h(t-h),\eta^h(t-h))\\
&=\la\bn,\psi(\lv,\lx,\li)\ra.
\end{split}
\end{align}

%	\begin{remark}\rm
%		\textcolor{red}{
%		Because of the way the minimization scheme is constructed and Lemma \ref{lemma.weak.minors_TIME_DERIVATIVES}, 
%		the constraints $\xi^B=\Phi^B(F)$ are inherited in the $h \to 0$ limit. We may think of the Young measure 
%		generated by the sequence $\{F^h\}_h$ as a marginal of the measure $\bn.$ Recall that
%		$$\bn\in L^{\infty}_{weak}(Q_{T};\mathcal{P}(\mathbb{R}^3\times\mathbb{M}^{3\times 3}\times\mathbb{M}^{3\times 3}\times\mathbb{R}\times\mathbb{R})),$$
%		so its marginal $\bn^F$ will be an element of the space
%		$$\bn^F\in L^{\infty}_{weak}(Q_{T};\mathcal{P}(\mathbb{R}^3))$$ 
%		such that
%		\begin{align}
%		\label{nu^F}
%		\Phi(\la\bn^F,\lf\ra)=\la\bn^F,\Phi(\lf)\ra=\la\bn,(\lf,\lz,\lw)\ra=(F,\mathrm{cof}F,\det F),\quad(x,t)\:a.e..
%		\end{align}
%		Equation \eqref{nu^F} is induced by the weak continuity of the determinant and cofactor functions, while because of
%		\eqref{dt_Z,W_convergence} we can deduce that for the functions 
%%		namely
%%		$$\xi=\Phi(\la\bn^F,\lf\ra)=\Phi(F)\quad(x,t)\:a.e..$$
%		which are representable as in \eqref{psi_weak-lim_INXI} there holds
%		\begin{align*}
%		\la\bn,\psi(v,\xi,\eta)\ra=\la\bn,\psi(v,\Phi(F),\eta)\ra \quad(x,t)\:a.e..
%		\end{align*}}
%	\end{remark}

	Equation \eqref{h_limit_F-cofF-detF} implies
	\begin{align}
	\la\bn , \Phi(\lf)\ra= \Phi (\la \bn, \lf \ra )= (F,\mathrm{cof}F,\det F) \, .
	\end{align}
	Because of the way the minimization scheme is constructed and Lemma \ref{lemma.weak.minors_TIME_DERIVATIVES}, 
	the constraints $\xi^B=\Phi^B(F)$ are inherited in the $h \to 0$ limit.
	Finally, by \cite[Lemma 4]{MR1831179}, the transport identities \eqref{transiden} are weakly continuous and \eqref{e:mv.sol.xi-v-eta}$_1$
	holds.

We continue the proof by writing equations \eqref{lemma_weak_formulation},\eqref{E-L_min} and \eqref{weak_enery_inequality}
for the iterates $(v^j,\xi^j,\eta^j)$ and $(v^{j-1},\xi^{j-1},\eta^{j-1}),$ namely
\begin{align*}
\int\phi\frac{1}{h}(\xi^j-\xi^{j-1})^B\:dx&=-\int\phzj v_i^j\partial_{\alpha}\phi\:dx\\
\int\phi\frac{1}{h}(\eta^j-\eta^{j-1})\:dx&=\int\phi\frac{r}{\hat{\theta}(\xi^{j-1},\eta^{j-1})}\:dx\\
\int\phi\frac{1}{h}(v_i^j-v_i^{j-1})\:dx&=-\int\pgej\phzj\partial_{\alpha}\phi\:dx
\end{align*}
and
\begin{small}
	\begin{align*}
	%\label{weak_enery_inequality}
	%\begin{split}
	\int\phi\frac{1}{h}\left(\frac{1}{2}|v^j|^2+\right.&\left.\hat{e}(\xi^j,\eta^j)-\frac{1}{2}|v^{j-1}|^2-\hat{e}(\xi^{j-1},\eta^{j-1})\right)\:dx\leq\\
	&\leq-\int\left(\pgej\phzj v_i^j\right)\partial_{\alpha}\phi\:dx
	+\int\phi\:\hat{\theta}(\xi^j,\eta^j)\frac{r}{\hat{\theta}(\xi^{j-1},\eta^{j-1})}\:dx.
	%\end{split}
	\end{align*}
\end{small}
Observing that
\begin{align}
\label{dt-V-Xi-Eta}
\partial_tV^h=\frac{1}{h}(v^h-v^h(t-h)),\quad\partial_t\Xi^h=\frac{1}{h}(\xi^h-\xi^h(t-h)),\quad\partial_tH^h=\frac{1}{h}(\eta^h-\eta^h(t-h)),
\end{align}
we can rewrite \eqref{discrete_thermoelasticity}, \eqref{dicrete_energy_ineq} in terms of the sequences
$(V^h,\Xi^h,H^h),$ $(v^h,\xi^h,\eta^h)$ and $(v^h(t-h),\xi^h(t-h),\eta^h(t-h)):$
\begin{align}
\label{weakform:dtXi}
\int_{Q_{T}}\phi\partial_t(\Xi^h)^B\:dxdt&=-\int_{Q_{T}}\frac{\partial\Phi^B}{\partial F_{i\alpha}}(f^h(t-h))v_i^h\partial_{\alpha}\phi\:dxdt\\
\label{weakform:dteta}
\int_{Q_{T}}\phi\partial_tH^h\:dxdt&=\int_{Q_{T}}\phi\frac{r}{\hat{\theta}(\xi^h(t-h),\eta^h(t-h))}\:dxdt\\
\label{weakform:dtv}
\int_{Q_{T}}\phi\partial_tV^h\:dxdt
&=-\int_{Q_{T}}\frac{\partial\hat{e}}{\partial\xi^B}(\xi^h,\eta^h)
\frac{\partial\Phi^B}{\partial F_{i\alpha}}(f^h(t-h))\partial_{\alpha}\phi\:dxdt
\end{align}
where \eqref{weakform:dtXi} is a vector relation that splits into the following:
\begin{align*}
\int_{Q_{T}}\phi\partial_tF_{i\alpha}^h\:dxdt&=-\int_{Q_{T}}v_i^h\partial_{\alpha}\phi\:dxdt\\
\int_{Q_{T}}\phi\partial_tZ_{k\gamma}^h\:dxdt
&=-\int_{Q_{T}}\epsilon_{ijk}\epsilon_{\alpha\beta\gamma}f_{j\beta}^h(t-h)v_i^h\partial_{\alpha}\phi\:dxdt\\
\int_{Q_{T}}\phi\partial_tW^h\:dxdt
&=-\int_{Q_{T}}\bigl((\mathrm{cof}f^h(t-h)\bigr)_{i\alpha}v_i^h\partial_{\alpha}\phi\:dxdt.
\end{align*}
Finally the energy inequality becomes
\begin{align}
\label{weakform:energy}
\begin{split}
\int_{Q_{T}}\phi\partial_t\left(\frac{1}{2}|V^h|^2+\right.&\left.\vphantom{\frac{1}{2}}\hat{e}(\Xi^h,H^h)\right)\:dxdt\\
&\leq-\int_{Q_{T}}\left(\frac{\partial\hat{e}}{\partial\xi^B}(\xi^h,\eta^h)
\frac{\partial\Phi^B}{\partial F_{i\alpha}}(f^h(t-h)) v_i^h\right)\partial_{\alpha}\phi\:dxdt\\
&+\int_{Q_{T}}\phi\:\hat{\theta}(\xi^h,\eta^h)\frac{r}{\hat{\theta}(\xi^h(t-h),\eta^h(t-h))}\:dxdt
\end{split}
\end{align}
for any $0<\phi=\phi(x,t)\in C_c^{\infty}(\mathbb{T}^3\times[0,T)).$

The derivation of \eqref{weakform:dtXi}, \eqref{weakform:dteta} and \eqref{weakform:dtv} from \eqref{discrete_thermoelasticity} and 
\eqref{dt-V-Xi-Eta} is straightforward,. By contrast, the derivation of \eqref{weakform:energy} is a nontrivial variant of retrieving
\eqref{dicrete_energy_ineq} from \eqref{discrete_thermoelasticity}. Let $I(V^h,\Xi^h,H^h)$ as in \eqref{I(U)}, then 
using \eqref{weakform:dtXi}, \eqref{weakform:dteta} and \eqref{weakform:dtv} we can calculate
\begin{align*}
\partial_t&\left(\frac{1}{2}|V^h|^2+\hat{e}(\Xi^h,H^h)\right)=\\
&=V_i^h\partial_t V_i^h+\frac{\partial\hat{e}}{\partial\xi^B}(\Xi^h,H^h)\partial_t(\Xi^h)^B+\frac{\partial\hat{e}}{\partial\eta}(\Xi^h,H^h)\partial_tH^h\\
&=V_i^h\frac{v_i^h-v_i^h(t-h)}{h}+\frac{\partial\hat{e}}{\partial\xi^B}(\Xi^h,H^h)\frac{(\xi^h-\xi^h(t-h))^B}{h}
+\frac{\partial\hat{e}}{\partial\eta}(\Xi^h,H^h)\frac{\eta^h-\eta^h(t-h)}{h}\,.
\end{align*}
\tcb{
	Since $I(v,\xi, \eta)$ is a convex function and $(V^h,\Xi^h,H^h)$ is by its definition a convex combination of $(v^h,\xi^h,\eta^h)$ 
	and $(v^h(t-h),\xi^h(t-h),\eta^h(t-h))$, we have the inequality
	$$
	\begin{aligned}
	&V_i^h\frac{v_i^h-v_i^h(t-h)}{h}+\frac{\partial\hat{e}}{\partial\xi^B}(\Xi^h,H^h)\frac{(\xi^h-\xi^h(t-h))^B}{h}
	+\frac{\partial\hat{e}}{\partial\eta}(\Xi^h,H^h)\frac{\eta^h-\eta^h(t-h)}{h}
	\\
	&\le v_i^h\frac{v_i^h-v_i^h(t-h)}{h}+\frac{\partial\hat{e}}{\partial\xi^B}(\xi^h,\eta^h)\frac{(\xi^h-\xi^h(t-h))^B}{h}+\hat{\theta}(\xi^h,\eta^h)\frac{\eta^h-\eta^h(t-h)}{h}
	\end{aligned}
	$$
	Next, proceeding as in the derivation of \eqref{lemma2_calculation}, we conclude
	\begin{small}
		\begin{align*}
		\partial_t&\left(\frac{1}{2}|V^h|^2+\hat{e}(\Xi^h,H^h)\right)\\
		&\le v_i^h\frac{v_i^h-v_i^h(t-h)}{h}+\frac{\partial\hat{e}}{\partial\xi^B}(\xi^h,\eta^h)\frac{(\xi^h-\xi^h(t-h))^B}{h}+\hat{\theta}(\xi^h,\eta^h)\frac{\eta^h-\eta^h(t-h)}{h}\\
		&=\partial_{\alpha}\left(\frac{\partial\hat{e}}{\partial\xi^B}(\xi^h,\eta^h)\frac{\partial\Phi^B}{\partial F_{i\alpha}}(f^h(t-h)) v_i^h\right)
		+\hat{\theta}(\xi^h,\eta^h)\frac{r}{\hat{\theta}(\xi^h(t-h),\eta^h(t-h))}.
		\end{align*}
	\end{small}
}

This leads to \eqref{weakform:energy}, which we can now rewrite for a test function
depending solely on time, namely $0\leq\varphi(t) \in C_c^{\infty}([0,T));$ as a result, the divergence
term integrates to zero and we obtain
\begin{align}
\label{dissipative_weakform:energy}
\begin{split}
-\int_{\mathbb{T}^3}\varphi(0)\left(\frac{1}{2}|V^h|^2+\hat{e}(\Xi^h,H^h)\right)(x,0)\:dx
&-\int_{Q_{T}}\left(\frac{1}{2}|V^h|^2+\hat{e}(\Xi^h,H^h)\right)\varphi'(t)\:dxdt\\
\leq&\int_{Q_{T}}\varphi(t)\:\hat{\theta}(\xi^h,\eta^h)\frac{r}{\hat{\theta}(\xi^h(t-h),\eta^h(t-h))}\:dxdt.
\end{split}
\end{align}
%-\int& \varphi(0)\left(\left\la\bn,\frac{1}{2}|\lv|^2+\hat{e}(\lx,\li)\right\ra (x,0)\:dx +\bg_0(dx) \right)\\
%&-\int_0^T
%\int  \varphi^{\prime}(t) \left(\left\la\bn,\frac{1}{2}|\lv|^2+\hat{e}(\lx,\li) \right\ra(x,t)  \, dx\:dt +\bg(dx\:dt) \right)  \\
%&\leq\int_0^T\int \left\la\bn,r\right\ra\varphi(t)\:dx\:dt,

First, we pass to the limit in \eqref{weakform:dteta} and \eqref{weakform:dtv}
to retrieve equations \eqref{e:mv.sol.xi-v-eta}$_2$ and \eqref{e:mv.sol.xi-v-eta}$_3$.
We do this using the well-known Theorem of Young Measure representation in the $L^p$ 
setting given in \cite{MR1036070}. 
Starting with \eqref{weakform:dteta}, we need to make sure that the term
\begin{align}
\label{tem_r_and_theta}
\frac{r}{\hat{\theta}(\xi^h(t-h),\eta^h(t-h))}
\end{align}
is representable. The function $r=r(x,t)$ the (external) radiative heat supply is assumed in $L^{\infty}(Q_{T})$ 
and using the condition \eqref{growth.con.2} and the assumption that
$\hat{\theta}\geq\delta>0$ -- 
in order to avoid the denominator approaching zero -- we obtain
\begin{align*}
\frac{r}{\hat{\theta}(\xi^h(t-h),\eta^h(t-h))}
\rightharpoonup\left\la\bn,\frac{r}{\hat{\theta}(\Phi(\lambda_F),\li)}\right\ra,
\qquad\text{weakly in }L^1(Q_{T})
\end{align*}
and derive \eqref{e:mv.sol.xi-v-eta}$_2$. 

Moving to \eqref{weakform:dtv}, in the limit we expect
\begin{align*}
\int\left\la\bn,\lvi\right\ra(x,0)\phi(x,0)\:dx
+\int_{Q_{T}}&\left\la\bn,\lvi\right\ra\partial_t\phi\:dxdt\\
&=\int_{Q_{T}}\left\la\bn,\frac{\partial\hat{e}}{\partial\xi^B}(\Phi(\lambda_F),\lambda_{\eta})
\frac{\partial\Phi^B}{\partial F_{i\alpha}}(\lf)\right\ra
\partial_{\alpha}\phi\:dxdt.
\end{align*}	
In order to pass to the limit, we need to examine whether the term
\begin{align}
\label{e:mv.sol.v-TERM}
\frac{\partial\hat{e}}{\partial\xi^B}(\xi^h,\eta^h)\frac{\partial\Phi^B}{\partial F_{i\alpha}}(f^h(t-h))
\end{align}
is representable. Indeed, since the sequences $(\xi^h,\eta^h)$ and $f^h(t-h)$ obey the uniform bounds
	\eqref{h-sequenceses-boundness} the argument in  \eqref{bound_p_prime} together with 
	\eqref{growth.con.3}  imply the term in \eqref{e:mv.sol.v-TERM} is uniformly bounded in $L^\infty ( L^{p'})$.
%	\begin{align*}
%	\left|\frac{\partial\hat{e}}{\partial\xi^B}\right.&\left.
%	(\xi^h,\eta^h)\frac{\partial\Phi^B}{\partial F_{i\alpha}}(f^h)
%	-\frac{\partial\hat{e}}{\partial\xi^B}(\xi^h,\eta^h)\frac{\partial\Phi^B}{\partial F_{i\alpha}}(f^h(t-h))\right|\\
%	&\!\!\!\!\!\!\!\!
%	\leq\left(|\partial_{\zeta}\hat{e}|+|\partial_w\hat{e}|(|f^h|+|f^h(t-h)|)\right)|f^h-f^h(t-h)|\\
%	&\!\!\!\!\!\!\!\!
%	\leq c\left((|f^h|^p+|\zeta^h|^q+|w^h|^{\rho}+|\eta^h|^{\ell}+1)^{\frac{p-2}{p}}+(|f^h|+|f^h(t-h)|)^{p-2}\right)
%	|f^h-f^h(t-h)|
%	\end{align*}
%	and because of \eqref{limit-Lip-psi} we have
%	\begin{align*}
%	\left|\int_{Q_{T}}\left[
%	\frac{\partial\hat{e}}{\partial\xi^B}(\xi^h,\eta^h)\frac{\partial\Phi^B}{\partial F_{i\alpha}}(f^h)
%	-\frac{\partial\hat{e}}{\partial\xi^B}(\xi^h,\eta^h)\frac{\partial\Phi^B}{\partial F_{i\alpha}}(f^h(t-h))
%	\right]\phi\:dxdt\right|\to 0,
%	\end{align*}
Therefore in the limit $h\to 0$
\begin{align*}
\frac{\partial\hat{e}}{\partial\xi^B}(\xi^h,\eta^h)\frac{\partial\Phi^B}{\partial F_{i\alpha}}(f^h(t-h))
\rightharpoonup
\left\la\bn,\frac{\partial\hat{e}}{\partial\xi^B}(\Phi(\lf),\li)\frac{\partial\Phi^B}{\partial F_{i\alpha}}(\lf)\right\ra,
\quad\text{weakly in }L^1(Q_{T}) \, ,
\end{align*}
in turn leading to \eqref{e:mv.sol.xi-v-eta}$_3$. 

Finally we pass to the limit in \eqref{weakform:energy} to get a dissipative measure-valued solution:
\begin{align*}
-\int& \varphi(0)\left\la\bn,\frac{1}{2}|\lv|^2+\hat{e}(\Phi(\lambda_F),\li)\right\ra (x,0)\:dx \\ %+\bg_0(dx) \right)\\
&-\int_0^T
\int  \varphi^{\prime}(t) \left(\left\la\bn,\frac{1}{2}|\lv|^2+\hat{e}(\Phi(\lambda_F),\li) \right\ra(x,t)  \, dx\:dt +\bg(dx\:dt) \right)  \\
&\leq\int_0^T\int \left\la\bn,r\right\ra\varphi(t)\:dx\:dt\,.
\end{align*}
%First, the function $\ds(x,t)\mapsto\left\{\frac{1}{2}|v|^2+\hat{e}(\xi,\eta)\right\}\:dxdt$ is sequentially 
%weakly relatively compact
%in the space of nonnegative Radon measures $\mathcal{M}^+(Q_{\infty}),$ but not in $L^1(Q_{\infty}),$ so that
%the Fundamental Lemma on Young measures does not apply. 
Since we know that the functions $(V^h,\Xi^h,H^h)$
are all bounded in some $L^p$ space -because of \eqref{h-sequenceses-boundness} and \eqref{growth.con.1}- one
could apply the generalized Young measure Theorem \cite{dm87,ab97}, in order to pass to the limit, since the only 
assumption required is $L^1$ boundedness. The Theorem asserts that given a sequence of functions 
$\{u_n\},$ $u_n:\mathbb{T}^d\to\mathbb{R}^m,$ bounded in $L^p(\mathbb{T}^d),$ ($p\geq 1$) there exists a 
subsequence (which we will not relabel), a parametrized family of probability measures 
$\bn\in L^{\infty}_{weak}(\mathbb{T}^d;\mathcal{P}(\mathbb{R}^m)),$ a nonnegative measure 
$\boldsymbol{\mu}\in\mathcal{M}^+(\mathbb{T}^d)$ and a parametrized probability measure on a 
sphere $\bn^{\infty}\in L^{\infty}_{weak}((\mathbb{T}^d,\boldsymbol{\mu});\mathcal{P}(S^{m-1}))$ such that
\begin{align}
\label{ab97_weaklimits}
\psi(x,u_n)\:dx\rightharpoonup\int_{\mathbb{R}^m}\psi(x,\lambda)\:d\boldsymbol{\nu}dx
+ \int_{S^{m-1}}\psi^{\infty}(x,z)\:d\boldsymbol{\nu}^{\infty}dx
\quad\text{weakly-$\ast$,}
\end{align}
for all $\psi$ continuous with well-defined recession function 
$$\psi^{\infty}(x,z):=\lim_{\substack{s\to\infty\\z'\to z}}\frac{\psi(x,sz')}{s^p}.$$
Now since the sequences $(V^h,\Xi^h,H^h),$ $(v^h,\xi^h,\eta^h)$ and $(v^h(t-h),\xi^h(t-h),\eta^h(t-h))$ are bounded 
in different spaces and have different growth, we need to apply a refinement of the aforementioned Theorem as, for instance, in \cite{GSW2015}: 
Consider a sequence of maps $u_n=\{u^1_n,u^2_n\}$ where $\{u^1_n\}$ is bounded in some 
$L^p(\mathbb{T}^d;\mathbb{R}^b)$ and $\{u^2_n\}$ is bounded in $L^q(\mathbb{T}^d;\mathbb{R}^l).$ Define the 
non-homogeneous unit sphere 
$$
S^{b+l-1}_{pq}:=\{(\beta_1,\beta_2)\in\mathbb{R}^{b+l}:|\beta_1|^{2p}+|\beta_2|^{2q}=1\}\,.
$$
\tcr{
	Then one can pass to the limit as in \eqref{ab97_weaklimits} where 
	$$
	\begin{aligned}
	\psi^{\infty}(x,z) 
	&:=\lim_{\substack{x'\to x\\s\to\infty\\(\beta_1',\beta_2')\to (\beta_1,\beta_2)}}
	\frac{\psi(x',s^q\beta_1',s^p\beta_2')}{s^{pq}}
	= \lim_{\substack{x'\to x\\ \tau \to\infty\\(\beta_1',\beta_2')\to (\beta_1,\beta_2)}}
	\frac{\psi(x', \tau^{\frac{1}{p}} \beta_1', \tau^{\frac{1}{q}} \beta_2')}{\tau}.
	\end{aligned}
	$$
	When $p, q > 1$ one may replace  $S^{b+l-1}_{pq}$ by 
	${S'}^{b+l-1}_{pq} = \{(\beta_1,\beta_2)\in\mathbb{R}^{b+l}:|\beta_1|^{p}+|\beta_2|^{q}=1\}$.
}

	We apply the generalized Young measure theorem \cite{dm87,ab97} and the idea of generalized balls from \cite{GSW2015}
	in the present context of uniform energy bounds for  $(V^h,\Xi^h,H^h)$:
	\begin{align*}
	\int\left(\frac{1}{2}|V^h|^2+\hat{e}(\Xi^h,H^h)\right)\:dx \le C \, .
	\end{align*}
	The relevant $L^r$ norms are dictated by \eqref{growth.con.1} and have different exponents in
	different directions. We define the generalized sphere
	$$
	S^{22} = \{ (F, \zeta, w, \eta, v) \in \mathbb{R}^{23} : |F|^p + |\zeta|^q + |w|^\rho + |\eta|^\ell + |v|^2 = 1 \} \, .
	$$
	The form of the recession function for the energy follows from \cite[Thm 2.5]{ab97} via the procedure described in
	\cite[Sec 5.4.1]{christoforou2016relative}
	The recession function reads
	\begin{align*}
	\left(\frac{1}{2}|v|^2+\hat{e}(\xi,\eta)\right)^{\infty}=\lim_{\tau \to\infty}
	\left(\frac{1}{2}|v|^2+\frac{\hat{e}( \tau^{\frac{1}{p}} F ,  \tau^{\frac{1}{q}} \zeta , \tau^{\frac{1}{\rho}} w,  \tau^{\frac{1}{\ell}} \eta)}{\tau}\right),
	\end{align*}
	and we require it to be continuous on $S^{22}$.  Then, along a subsequence in $h$,
	\begin{align*}
	\frac{1}{2}|V^h|^2+\hat{e}(\Xi^h,H^h) \overset{\ast}{\rightharpoonup}
	\left\la\bn,\frac{1}{2}|\lv|^2+\hat{e}(\Phi(\lf),\li)\right\ra \, dx
	+ \left\la\bn^{\infty},\left(\frac{1}{2}|\lv|^2+\hat{e}(\Phi(\lf),\li)\right)^{\infty}\right\ra\boldsymbol{\mu}
	\end{align*}
	weak-$\ast$ in the sense of measures, where $\bn\in\mathcal{P}(Q_{T};\mathbb{R}^{23}),$ 
	$\bn^{\infty}\in\mathcal{P}((Q_{T},\boldsymbol{\mu});S^{22})$ and $\boldsymbol{\mu}\in\mathcal{M}^+(Q_T).$ 
	Note that
	\eqref{growth.con.1} implies $\ds\left(\frac{1}{2}|\lv|^2+\hat{e}(\Phi(\lf),\li)\right)^{\infty}>0$
	so that the concentration measure $\gamma \in \mathcal{M}^+(Q_T)$ is nonnegative,
	\begin{align}
	\label{concmeas}
	\bg:=\left\la\bn^{\infty},\left(\frac{1}{2}|\lv|^2+\hat{e}(\Phi(\lf),\li)\right)^{\infty}\right\ra\boldsymbol{\mu} \ge 0 \, .
	\end{align}

In summary, we have constructed a measure-valued solution that satisfies in the sense of distributions
\begin{align*}
%\label{limiting_eqns1}
\partial_t \Phi^B(F)&=\partial_{\alpha}\left(\phv \right) \\ 
%\label{limiting_eqns2}
\partial_t \eta &= \Big\la\bn,\frac{r}{\hat{\theta}(\Phi(\lf),\li)}\Big\ra\\
%\label{limiting_eqns3}
\partial_t v_i
&=\partial_{\alpha}\left\la\bn,\frac{\partial \hat{e}}{\partial\xi^B}(\Phi(\lf),\lambda_{\eta})\frac{\partial\Phi^B}{\partial F_{i\alpha}}(\lf)\right\ra
\end{align*}
and for  $\varphi(t) \in C^{\infty}_c([0,T))$, $\varphi(t)\ge 0$,
\begin{align*}
%\label{limiting_energy eqn}
%\begin{split}
-\int& \varphi(0)\left(\Big\la\bn,\frac{1}{2}|\lv|^2+\hat{e}(\Phi(\lf),\li)\Big\ra (x,0)\:dx +\bg_0(dx) \right)\\
&-\int_0^T
\int  \varphi^{\prime}(t) \left(\Big\la\bn,\frac{1}{2}|\lv|^2+\hat{e}(\Phi(\lf),\li) \Big\ra(x,t)  \, dx\:dt +\bg(dx\:dt) \right)  \\
&\leq\int_0^T\int  r(x,t) \varphi(t)\:dx\:dt \, .
%\end{split}
\end{align*}

	\begin{remark}\rm
		The representation theory of \cite{ab97} gives a precise form of the concentration measure \eqref{concmeas}
		but requires continuity of the recession function. Instead, one can use a hands-on construction,
		carried out in \cite[App]{MR1831175}, that provides less information on the concentration measure but
		requires only that $\hat e(\xi, \eta)$ is strictly convex and positive.
		This can be applied in the present context, since all the relevant functions can be represented by classical Young measures
		and the only function that may have concentrations is the energy and yields a positive concentration measure.	
	\end{remark}

\section{Convergence of the scheme in the smooth regime}\label{3:SEC6}
%	{e:mv.sol.xi-v-eta}
%{e:mv.sol.Fvt.energy}
In this Section we compare the measure-valued solutions constructed in the previous Section and satisfying equations 
\eqref{e:mv.sol.xi-v-eta},\eqref{e:mv.sol.Fvt.energy} against a
strong solution for polyconvex thermoelasticity via the relative entropy method. The goal is to show that the solutions constructed
via the variational scheme converge to the solution of \eqref{e:adiabatic_thermoelasticity} so long as the latter is smooth.
Consider the Lipschitz solution $(\Phi(\bar{F}),\bar{v},\bar{\eta})$ defined on $[0,T] \times \mathbb{T}^3$ and solving 
\begin{align}
\label{e:strong_system}
\begin{split}
\partial_t \Phi^B(\bar{F})&=\partial_{\alpha}\left(\phbv \right)\\ 
\partial_t\bar{\eta} &=\frac{r}{\hat{\theta}(\Phi(\bar{F}),\bar{\eta})}\\
\partial_t \bar{v}_i&=\partial_{\alpha}\left(\frac{\partial \hat{e}}{\partial\xi^B}(\Phi(\bar{F}),\bar{\eta})\phb\right)\\
\partial_t \left(\frac{1}{2}|\bar{v}|^2+\hat{e}(\Phi(\bar{F}),\bar{\eta})\right)
&=\partial_{\alpha}\left(\frac{\partial\hat{e}}{\partial\xi^B}(\Phi(F),\bar{\eta})\phb \bar{v}_i\right)+r .
\end{split}
\end{align}
We assume the initial data have no concentration $\bg_0=0.$
Next we write the weak form of the difference between \eqref{e:mv.sol.xi-v-eta}$_{1,3}$,\eqref{e:mv.sol.Fvt.energy} and \eqref{e:strong_system}$_{1,3,4}$,
tested against the functions 
\begin{align*}
-\hat\theta(\Phi(\bar{F}),\bar{\eta})G(\bar U)\varphi(t) = \left(-\frac{\partial \hat{e}}{\partial\xi^B}(\Phi(\bar{F}),\bar{\eta})\, , -\bar{v} ,1 \right)\varphi(t) 
\end{align*}
where $0<\varphi(t)\in C^{\infty}_c([0,T))$ namely
\begin{align}
\begin{split}
\label{e:eq1phi.wk.mv}
\int &\left(-\frac{\partial\hat{e}}{\partial\xi^B}(\Phi(\bar{F}),\bar{\theta})(\Phi^B(F)-\Phi^B(\bar{F}))\right)(x,0)\varphi(0)\:dx \\
&+\int_0^T\int\left(-\frac{\partial\hat{e}}{\partial\xi^B}(\Phi(\bar{F}),\bar{\theta})(\Phi^B(F)-\Phi^B(\bar{F}))\right)\varphi^{\prime}(t)\:dx\:dt\\
&=\int_0^T\int \left[\partial_t\Big(\frac{\partial\hat{e}}{\partial\xi^B}(\Phi(\bar{F}),\bar{\theta})\Big)(\Phi^B(F)-\Phi^B(\bar{F}))\right.\\
&\qquad\qquad\left.-\partial_{\alpha}\Big(\frac{\partial\hat{e}}{\partial\xi^B}(\Phi(\bar{F}),\bar{\theta})\Big)\Big(\phv-\phbv\Big)\right]\varphi(t) dxdt \, ,\end{split}
\end{align}
\begin{equation}
\begin{split}
\label{e:eq2phi.wk.mv}
\int(&-\bar{v}_i(\left\la\bn,\lvi\right\ra-\bar{v}_i))(x,0)\varphi(0)\:dx 
+\int_0^T\int-\bar{v}_i(\left\la\bn,\lvi\right\ra-\bar{v}_i)\varphi^{\prime}(t)\:dx\:dt\\
&=-\int_0^T\int\left[-\partial_{\alpha}\left(\frac{\partial\hat{e}}{\partial\xi^B}(\Phi(\bar{F}),\bar{\eta})\phb\right)(\left\la\bn,\lvi\right\ra-\bar{v}_i)\right.\\
&\:\:\hspace{2cm}+\partial_{\alpha}\bar{v}_i\left(\left\la\bn,\frac{\partial\hat{e}}{\partial\xi^B}(\Phi(\lambda_{F}),\lambda_{\eta})\phl\right\ra\right.\\
&\:\:\hspace{4cm}\left.\left.-\frac{\partial\hat{e}}{\partial\xi^B}(\Phi(\bar{F}),\bar{\eta})\phb\right)\right]\varphi(t)\:dx\:dt\;,
\end{split}
\end{equation}
and 
\begin{align}
\label{e:eq3phi.wk.mv}
\int&\left(\left\la\bn,\frac{1}{2}|\lv|^2+\hat{e}(\Phi(\lf),\li)\right\ra
-\frac{1}{2}|\bar{v}|^2-\hat{e}(\Phi(\bar{F}),\bar{\theta})\right)\!(x,0)\;\varphi(0)\:dx\nonumber \\
&%\quad
+\int_0^T \int \Big\{\left(\left\la\bn,\frac{1}{2}|\lv|^2+\hat{e}(\Phi(\lf),\li)\right\ra
-\frac{1}{2}|\bar{v}|^2-\hat{e}(\Phi(\bar{F}),\bar{\theta})\right)
%\nonumber\\  &\:\:\hspace{2.5cm}
+\bg\Big\}\varphi^{\prime}(t)\:dx\:dt\nonumber \\
&\geq-\int_0^T \int (\left\la\bn,r\right\ra-\bar{r})\varphi(t)\:dx\:dt,
\end{align}
while testing the difference between \eqref{e:mv.sol.xi-v-eta}$_{2}$ and \eqref{e:strong_system}$_{2}$ against $\theta(\Phi(\bar{F}),\bar{\eta})\varphi(t)$ we get
\begin{align}
\begin{split}
\label{e:entr2.wk.mv}
-\int\theta&(\Phi(\bar{F}),\bar{\eta})(\left\la\bn,\li\right\ra-\hat{\eta})(x,0)\varphi(0)\:dx\\
&-\int_0^T\int\theta(\Phi(\bar{F}),\bar{\eta})(\left\la\bn,\li\right\ra-\hat{\eta})\varphi^{\prime}(t)\:dx\:dt\\
&=\int_0^T\int \Bigg[\partial_t\theta(\Phi(\bar{F}),\bar{\eta})(\left\la\bn,\li\right\ra-\hat{\eta})\\
&\qquad\quad\qquad
+\theta(\Phi(\bar{F}),\bar{\theta})\left(\left\la\bn,\frac{r}{\bar{\theta}(\Phi(\lf),\li)}\right\ra-\frac{r}{\bar{\theta}(\Phi(\bar{F}),\bar{\eta})}\right)\Bigg]\varphi(t)\:dx\:dt.
\end{split}
\end{align}
The formulation of the relative entropy inequality follows along the lines of the derivation in \cite[Section 4]{CGT2018}. 
The formal calculations are parallel, so here we omit the details. Having in mind \eqref{relative_I(U)}, we define
\begin{align*}
%\begin{split}
%\label{def.I}
\left\la\bn,I(\lambda_{U}|\bar{U})\right\ra&=\left\la\bn,I(\lv,\Phi(\lambda_{F}),\li|\bar{v},\Phi(\bar{F}),\bar{\eta})\right\ra\\
&:=\left\la\bn,\frac{1}{2}|\lv-\bar{v}|^2+\hat{e}(\Phi(\lf),\li|\Phi(\bar{F}),\bar{\eta})\right\ra
\end{align*} 
where
\begin{align*}
\hat{e}&(\Phi(\lf),\li|\Phi(\bar{F}),\bar{\eta})\\
&:=\hat{e}(\Phi(\lf),\li)-\hat{e}(\Phi(\bar{F}),\bar{\eta})-\frac{\partial\hat{e}}{\partial\xi^B}(\Phi(\bar{F}),\bar\eta)(\Phi^B(F)-\Phi^B(\bar{F}))
-\frac{\partial\hat{e}}{\partial\eta}(\Phi(\bar{F}),\bar\eta)(\li-\bar\eta),
\end{align*}
and we also define the following relative quantities:
\begin{align*}
&\left \la\bn,\hat{\theta}(\Phi(\lambda_{F}),\li|\Phi(\bar{F}),\bar{\eta})\right\ra
:=\Big \la\bn,\hat{\theta}(\Phi(\lambda_{F}),\li)-\hat{\theta}(\Phi(\bar{F}),\bar{\eta}) 
\\
&\qquad 
-\frac{\partial\hat{\theta}}{\partial\xi^B}(\Phi(\bar{F}),\bar\eta)(\Phi^B(F)-\Phi^B(\bar{F}))
-\frac{\partial\hat{\theta}}{\partial\eta}(\Phi(\bar{F}),\bar\eta)(\li-\bar\eta)\Big \ra
\end{align*} 
and
\begin{align*}
\left\la\bn,\frac{\partial\hat{e}}{\partial\xi^B}(\Phi(\lambda_{F}),\li|\Phi(\bar{F}),\bar{\eta})\right\ra
:=\left\la\bn,\frac{\partial\hat{e}}{\partial\xi^B}(\Phi(\lambda_{F})\li)-\frac{\partial\hat{e}}{\partial\xi^B}(\Phi(\bar{F}),\bar{\eta})\right.\\
\left.-\frac{\partial^2\hat{e}}{\partial\xi^B\partial\xi^A}(\Phi(\bar{F}),\bar\eta)(\Phi^B(F)-\Phi^B(\bar{F}))
-\frac{\partial^2\hat{e}}{\partial\xi^B\partial\eta}(\Phi(\bar{F}),\bar\eta)(\li-\bar\eta)\right\ra.
\end{align*} 
Then because of \eqref{theta-Phi} and employing the null-Lagrangian property \eqref{Euler-Lagrange} in conjunction with the findings of Lemma
\ref{Lemma_null_lagrangian_WEAK}, we can add together \eqref{e:eq1phi.wk.mv},\eqref{e:eq2phi.wk.mv},\eqref{e:eq3phi.wk.mv} and \eqref{e:entr2.wk.mv}
and derive the relative entropy inequality:
%\begin{small}
	\begin{align}
	\label{e:rel.en.id.mv}
	\begin{split}
	&\int\varphi(0)[\left\la\bn,I(\lambda_{U_0}|\bar{U_0})\right\ra\:dx]
	+\int_{0}^{T}\int\varphi'(t)\left[\left\la\bn,I(\lambda_{U}|\bar{U})\right\ra\:dx\:dt+\bg(dx\,dt)\right]\\
	&\geq-\int_{0}^{T}\int\varphi(t)\Bigg[-\partial_t\bar{\eta}\left\la\bn,\hat{\theta}(\Phi(\lf),\li|\Phi(\bar{F}),\bar{\eta})\right\ra\\
	&+\partial_t\Phi^B(\bar{F})\left\la\bn,\frac{\partial\hat{e}}{\partial \xi^B}(\Phi(\lf),\li|\Phi(\bar{F}),\bar{\eta})\right\ra\\
	&+\partial_{\alpha}\bar{v}_i\left\la\bn,\left(\frac{\partial\hat{e}}{\partial\xi^B}(\Phi(\lambda_{F}),\lambda_{\eta})
	-\frac{\partial\hat{e}}{\partial\xi^B}(\Phi(\bar{F}),\bar{\eta})\right)\left(\phl-\phb\right)\right\ra\\
	&+\partial_{\alpha}\left(\frac{\partial\hat{}}{\partial\xi^B}(\Phi(\bar{F}),\bar{\eta})\right)\!\left(\ph-\phb\right)\!\la\bn,(\lvi-\bar{v}_i)\ra\\
	&+\left\la\bn,\left(\frac{r}{\hat{\theta}(\Phi(\lambda_{F}),\li)}-\frac{r}{\hat{\theta}(\Phi(\bar{F}),\bar{\eta})}\right)
	(\hat{\theta}(\Phi(\lambda_{F}),\li)-\hat{\theta}(\Phi(\bar{F}),\bar{\eta}))\right\ra\Bigg]\:dx\,dt.
	\end{split}
	\end{align}
%\end{small}
Now having \eqref{e:rel.en.id.mv}, we can show that classical solutions are unique in the class of dissipative measure-valued solutions 
generated as limits of the discrete scheme. To prove this, we assert first that the following bounds on the relative entropy and the terms on the 
right hand side of \eqref{e:rel.en.id.mv}, can be obtained given the growth conditions \eqref{growth.con.1},\eqref{growth.con.2},
\eqref{growth.con.3} and due to the convexity of $\hat e,$ in the same manner as in \cite[Section 5]{CGT2018}:
\begin{lemma}
	Given the growth conditions \eqref{growth.con.1}-\eqref{growth.con.3},
	and if $\hat{e}\in C^3(\mathbb{R}^{19}\times[0,\infty)),$ and the smooth solution $(\bar{F},\bar{v},\bar{\eta})$
	lies in the compact set
	\begin{equation*}
	\Gamma_{M}:=\left\{(\bar{F},\bar{v},\bar{\eta}): |\bar{F}|\leq M, |\bar{v}|\leq M, |\bar{\eta}|\leq M\right\}
	\end{equation*}
	for a positive constant $M$ then:
	\begin{enumerate}
		\item[(I)]
		There exist $R=R(M)$ and constants $K_1=K_1(M,c)>0,\:K_2=K_2(M,c)>0$ such that
		\begin{small}
			\begin{align}
			\label{bound4}
			\begin{split}
			I(v,\Phi(F),\eta|\bar{v},\Phi(\bar{F}),\bar{\eta}) \geq
			\begin{dcases}
			\frac{K_1}{2}(|F|^p+\eta^{\ell}+|v|^2),& \text{if } |F|^p\!+\eta^{\ell}\!+|v|^2>R\\
			K_2(|\Phi(F)-\Phi(\bar{F})|^2 &\text{if } |F|^p\!+\eta^{\ell}\!+|v|^2\leq R\\
			\quad+|\eta-\bar{\eta}|^2\!+\!|v-\bar{v}|^2), & 
			\end{dcases}
			\end{split}
			\end{align}
		\end{small}
		for all $(\bar{F},\bar{v},\bar{\eta})\in\Gamma_{M}$.
		\item[(II)]
		There exist constants $C_1,C_2,C_3,C_4>0$ such that
		\begin{equation}
		\label{bound6}
		\begin{split}
		\Bigg|\left(\frac{\partial\Phi^B}{\partial F_{i\alpha}}(F)
		-\frac{\partial\Phi^B}{\partial F_{i\alpha}}(\bar{F})\right)
		&\left(\frac{\partial\hat{\psi}}{\partial\xi^B}(\Phi(F),\eta)
		-\frac{\partial\hat{\psi}}{\partial\xi^B}(\Phi(\bar{F}),\bar{\eta})\right)\Bigg|\leq\\
		&\leq C_1 I(v,\Phi(F),\eta|\bar{v},\Phi(\bar{F}),\bar{\eta}) \;,
		\end{split}
		\end{equation}
		\begin{equation}
		\label{bound7}
		\left|\frac{\partial\hat{e}}{\partial\xi}(\Phi(F),\eta|\Phi(\bar{F}),\bar{\eta})\right|\leq C_2 I(v,\Phi(F),\eta|\bar{v},\Phi(\bar{F}),\bar{\eta})\;,
		\end{equation}
		\begin{equation}
		\label{bound5}
		|\hat{\theta}(\Phi(F),\eta|\Phi(\bar{F}),\bar{\eta})|\leq C_3 I(v,\Phi(F),\eta|\bar{v},\Phi(\bar{F}),\bar{\eta})\;,
		\end{equation}
		and 
		\begin{equation}
		\label{bound8}
		\left|\left(\frac{\partial\Phi^B}{\partial F_{i\alpha}}(F)
		-\frac{\partial\Phi^B}{\partial F_{i\alpha}}(\bar{F})\right)(v_i-\bar{v}_i)\right|
		\leq C_4 I(v,\Phi(F),\eta|\bar{v},\Phi(\bar{F}),\bar{\eta})
		\end{equation}
		for all $(\bar{F},\bar{v},\bar{\eta})\in\Gamma_{M}$.
		
		\item[(III)] Under hypothesis \eqref{lowerb}, for some constant $C_5$, we have
		\begin{equation}
		\label{bound101}
		\Big | \Big(\frac{1}{\hat{\theta}(\Phi(F), \eta )}-\frac{1}{\hat{\theta}(\Phi(\bar{F}),\bar{\eta})}\Big)
		(\hat{\theta}(\Phi(F), \eta)-\hat{\theta}(\Phi(\bar{F}),\bar{\eta})) \Big | \le C_5 I(v,\Phi(F),\eta|\bar{v},\Phi(\bar{F}),\bar{\eta})
		\end{equation}
		for $(\bar{F},\bar{v},\bar{\eta})\in\Gamma_{M}$.
		
		\item[(IV)]
		There exist constants $K_1^{\prime},\:K_2^{\prime}$ and $R>0$ sufficiently large such that
	%	\begin{small}
			\begin{align}
			\label{bound9}
			\begin{split}
			I(v,\Phi(F),\eta|\bar{v},\Phi(\bar{F}),\bar{\eta}) \geq
			\begin{dcases}
			\frac{K_1^{\prime}}{4}(|F-\bar{F}|^p      & \text{if } |F|^p\!+\eta^{\ell}\!+|v|^2>R\\
			\quad+|\eta-\bar{\eta}|^{\ell}+|v-\bar{v}|^2),  &\\
			K_2^{\prime}(|\Phi(F)-\Phi(\bar{F})|^2 & \text{if } |F|^p\!+\eta^{\ell}\!+|v|^2\leq R\\
			\quad+\!|\eta-\bar{\eta}|^2\!+\!|v-\bar{v}|^2),   &\\
			\end{dcases}
			\end{split}
			\end{align}
%		\end{small}
		for all $(\bar{F},\bar{v},\bar{\eta})\in\Gamma_{M}$.

	\end{enumerate}
\end{lemma}

\begin{proof}
	The proofs  of (I), (II) and (IV) follow along the lines of \cite[Lemmas 5.1 and 5.2]{CGT2018}
	and  the details are omitted here. To prove (III), set
	$$
	J := \Big(\frac{1}{\hat{\theta}(\Phi(\lambda_{F}),\li)}-\frac{1}{\hat{\theta}(\Phi(\bar{F}),\bar{\eta})}\Big)
	(\hat{\theta}(\Phi(\lambda_{F}),\li)-\hat{\theta}(\Phi(\bar{F}),\bar{\eta}))
	$$
	and note that in the region $|F|^p\!+\eta^{\ell}\!+|v|^2\leq R$ we have 
	$$
	|J| \le C (|\Phi(F)-\Phi(\bar{F})|^2  +\!|\eta-\bar{\eta}|^2\!+\!|v-\bar{v}|^2) \, .
	$$
	On the complementary region $|F|^p\!+\eta^{\ell}\!+|v|^2 \ge R$,
	$$
	J = 2 - \frac{ \hat{\theta}(\Phi(\bar{F}),\bar{\eta}) }{\hat{\theta}(\Phi(F), \eta )}-\frac{\hat{\theta}(\Phi(F), \eta)}{\hat{\theta}(\Phi(\bar{F}),\bar{\eta})}
	$$
	and \eqref{growth.con.2}, \eqref{lowerb}, \eqref{bound4} imply \eqref{bound101}.
\end{proof}

\begin{theorem}
	\label{e:Uniqueness_thm.mv}
	Let $\bar{U} \in W^{1,\infty}(Q_T)$ be a Lipschitz solution of \eqref{e:strong_system} taking values in $\Gamma_{M}$ with initial data $\bar{U}^0$.
	Let $(\bn,\bg,U)$ be a dissipative measure-valued solution  satisfying \eqref{e:mv.sol.xi-v-eta},\eqref{e:mv.sol.Fvt.energy} for $r(x,t) \in L^\infty$
	under the constitutive assumptions \eqref{sigma-Phi},\eqref{theta-Phi} and with initial data 
	$\bn_{(0,x)} = \delta_{U_0(x)}$, for some $U^0$, with no concentration $\bg_0=0$. 
	\tcb{
		Suppose that
		$\nabla^2_{(\xi,\eta)}\hat e \ge c \, \mathbb{I} > 0$ and the growth conditions \eqref{growth.con.1},\eqref{growth.con.2}, \eqref{growth.con.3}
		hold for exponents $p\geq4,$ $q \ge 2$, $\rho > 1$ and  $\ell>1$. 
	}
	If  $U_0 (x) = \bar U_0 (x)$  then $\bn=\delta_{\bar{U}}$ and $U=\bar U$ a.e. on $Q_T.$
\end{theorem}

\begin{proof} 
	Let $\{\varphi_n\}$ be a sequence of $C^1$ monotone decreasing functions such that $\varphi_n\geq 0,$ for all $n\in\mathbb{N},$ converging as $n \to \infty$ to the 
	Lipschitz function
	\begin{align*}
	\varphi(\tau)=\begin{dcases}
	1 & 0\leq\tau\leq t\\
	\frac{t-\tau}{\varepsilon}+1 & t\leq\tau\leq t+\varepsilon\\
	0 & \tau\geq t+\varepsilon
	\end{dcases}
	\end{align*}
	for some $\varepsilon>0.$ Testing the relative entropy inequality~\eqref{e:rel.en.id.mv} against the functions $\varphi_n$ yields
		\begin{align}
		\label{rel.entr.wk-n.mv}
		\begin{split}
		&\int\varphi_n(0)\left\la\bn,I(\lambda_{U_0}|\bar{U_0})\right\ra\:dx
		+\int_{0}^{t}\int\varphi'_n(\tau)\left[\left\la\bn,I(\lambda_{U}|\bar{U})\right\ra\:dx\:d\tau+\bg(dx\,d\tau)\right]\\
		&\geq-\int_{0}^{t}\int\varphi_n(\tau)\Big[-\partial_t\bar{\eta}\left\la\bn,\hat{\theta}(\Phi(\lf),\li|\Phi(\bar{F}),\bar{\eta})\right\ra\\
		&+\partial_t\Phi^B(\bar{F})\left\la\bn,\frac{\partial\hat{e}}{\partial \xi^B}(\Phi(\lf),\li|\Phi(\bar{F}),\bar{\eta})\right\ra\\
		&+\partial_{\alpha}\bar{v}_i\left\la\bn,\left(\frac{\partial\hat{e}}{\partial\xi^B}(\Phi(\lambda_{F}),\lambda_{\eta})
		-\frac{\partial\hat{e}}{\partial\xi^B}(\Phi(\bar{F}),\bar{\eta})\right)\left(\phl-\phb\right)\right\ra\\
		&+\partial_{\alpha}\left(\frac{\partial\hat{}}{\partial\xi^B}(\Phi(\bar{F}),\bar{\eta})\right)\!\left(\ph-\phb\right)\!\la\bn,(\lvi-\bar{v}_i)\ra\\
		&+\left\la\bn,\left(\frac{r}{\hat{\theta}(\Phi(\lambda_{F}),\li)}-\frac{r}{\hat{\theta}(\Phi(\bar{F}),\bar{\eta})}\right)
		(\hat{\theta}(\Phi(\lambda_{F}),\li)-\hat{\theta}(\Phi(\bar{F}),\bar{\eta}))\right\ra\Big]\:dx\,d\tau.
		\end{split}
		\end{align}
%	\normalsize
	%	Passing to the limit as $n\to \infty$ we get
	%	\begin{small}
	%		\begin{align*}
	%	\int &\left\la\bn,I(\lv,\Phi(\lambda_{F}),\li|\bar{v},\Phi(\bar{F}),\bar{\eta})\right\ra(x,0)\:dx\\
	%&\quad
	%-\frac{1}{\varepsilon}\int_{t}^{t+\varepsilon}
	%\int \left[\left\la\bn,I(\lv,\Phi(\lambda_{F}),\li|\bar{v},\Phi(\bar{F}),\bar{\eta})\right\ra\:dx\:d\tau+
	%\bg(dxd\tau)\right]  \\ 
	%		&\geq- \int_0^{t+\varepsilon}\!\!\!\int \!\Big[-\partial_t\bar{\eta}\left\la\bn,\hat{\theta}(\Phi(\lf),\li|\Phi(\bar{F}),\bar{\eta})\right\ra\\
	%		&+\partial_t\Phi^B(\bar{F})\left\la\bn,\frac{\partial\hat{e}}{\partial \xi^B}(\Phi(\lf),\li|\Phi(\bar{F}),\bar{\eta})\right\ra\\
	%		&+\partial_{\alpha}\bar{v}_i\left\la\bn,\left(\frac{\partial\hat{e}}{\partial\xi^B}(\Phi(\lambda_{F}),\lambda_{\eta})
	%		-\frac{\partial\hat{e}}{\partial\xi^B}(\Phi(\bar{F}),\bar{\eta})\right)\left(\phl-\phb\right)\right\ra\\
	%		&+\partial_{\alpha}\left(\frac{\partial\hat{}}{\partial\xi^B}(\Phi(\bar{F}),\bar{\eta})\right)\!\left(\ph-\phb\right)\!\la\bn,(\lvi-\bar{v}_i)\ra\\
	%		&+\left\la\bn,\left(\frac{r}{\hat{\theta}(\Phi(\lambda_{F}),\li)}-\frac{r}{\hat{\theta}(\Phi(\bar{F}),\bar{\eta})}\right)
	%		(\hat{\theta}(\Phi(\lambda_{F}),\li)-\hat{\theta}(\Phi(\bar{F}),\bar{\eta}))\right\ra\Big]\:dx\,d\tau.
	%		\end{align*}
	%	\end{small}
	
	Passing first to the limit $n \to \infty$ using the fact that $\bg\geq 0$ and subsequently passing to the limit $\varepsilon \to 0^{+}$ using the 
	estimates (\ref{bound6}), (\ref{bound7}),~(\ref{bound5}),~\eqref{bound8} and \eqref{bound101}, we arrive at
	\begin{align*}
	\int  \left\la\bn,I(\lv,\Phi(\lf),\li  \right.|&\left.   \bar{v},\Phi(\bar{F}),\bar{\eta})\right\ra\:dx\:dt \leq\\
	\leq C  &\int_0^t\int \left\la\bn,I(\lv,\Phi(\lf),\li| \bar{v},\Phi(\bar{F}),\bar{\eta})\right\ra\:dx\:d\tau\\
	&+\int \left\la\bn,I(\lv,\Phi(\lf),\li| \bar{v},\Phi(\bar{F}),\bar{\eta})\right\ra(x,0)\:dx
	\end{align*}
	for $t\in (0,T).$ Note that the constant $C$ depends only on the smooth bounded solution $\bar{U}.$ 
	Then Gronwall's inequality implies 
	\begin{align*} 
	\int  \left\la\bn,I(\lv,\Phi(\lf),\li\right.&   |\left. \bar{v},\Phi(\bar{F}),\bar{\eta})\right\ra\:dx\:dt \leq\\
	&\leq C_1 e^{C_2t}\int \left\la\bn,I(\lv,\Phi(\lf),\li| \bar{v},\Phi(\bar{F}),\bar{\eta})\right\ra(x,0)\:dx.
	\end{align*}
	Now note that $\gamma_0 = 0$ and since  $\nu_{(x,0)} = \delta_{U_0(x)}$ and the initial data are the same, $U_0 = \bar U_0$ at $t=0$, the right hand side vanishes. 
	The proof follows by (\ref{bound9}).
\end{proof}

{\bf Acknowledgement.}
The authors thank the anonymous referee for very helpful comments that helped considerably in improving this work.



%\begin{thebibliography}{}
%%
%% and use \bibitem to create references. Consult the Instructions
%% for authors for reference list style.
%%

%\end{thebibliography}

\end{document}